\documentclass[dvips,bj]{imsart}

\RequirePackage[OT1]{fontenc}
\RequirePackage{amsthm,amsmath,natbib}
\usepackage{amssymb}
\arxiv{math.PR/0000000}

\startlocaldefs
\numberwithin{equation}{section}
\theoremstyle{plain}
\newtheorem{thm}{Theorem}[section]
\theoremstyle{plain}
\newtheorem{lem}{Lemma}[section]
\theoremstyle{remark}
\newtheorem{rem}{Remark}[section]
\theoremstyle{plain}
\newtheorem{ass}{Assumption}[section]
\endlocaldefs

\newcommand{\abs}[1]{\left\vert#1\right\vert}

\newcommand{\norm}[1]{\left\Vert#1\right\Vert}

\begin{document}

\begin{frontmatter}
\title{Posterior convergence rates in non-linear latent variable models}
\runtitle{Convergence rates in NL-LVM}

\begin{aug}
\author{\fnms{Debdeep} \snm{pati}\thanksref{e1}\ead[label=e1,mark]{dp55@stat.duke.edu}},
\author{\fnms{Anirban} \snm{Bhattacharya}\thanksref{e2}\ead[label=e2,mark]{ab179@stat.duke.edu}}
\and
\author{\fnms{David} \snm{Dunson}\thanksref{e3}\ead[label=e3,mark]{dunson@stat.duke.edu}}

\address{Department of Statistical Science, Box 90251, Duke University, Durham, NC 27708-0251, USA
\printead{e1,e2,e3}}


%
%

\end{aug}

\begin{abstract}
Non-linear latent variable models have become increasingly popular in a variety of applications. However, there has been little study on theoretical properties of these models. In this article, we study rates of posterior contraction in univariate density estimation for a class of non-linear latent variable models where unobserved $\mbox{U}(0,1)$ latent variables are related to the response variables via a random non-linear regression with an additive error. Our approach relies on characterizing the space of densities induced by the above model as kernel convolutions with a general class of continuous mixing measures. The literature on posterior rates of contraction in density estimation almost entirely focuses on finite or countably infinite mixture models. We develop approximation results for our class of continuous mixing measures. Using an appropriate Gaussian process prior on the unknown regression function, we obtain the optimal frequentist rate up to a logarithmic factor under standard regularity conditions on the true density.
\end{abstract}


\end{frontmatter}

\section{Introduction}
Kernel mixture models are known to be extremely flexible and have been extensively used for density estimation. Starting with a parametric kernel $\mathcal{K}(y, \theta)$, one can obtain a class of densities $f_G$ as
\begin{eqnarray}\label{eq:kernel}
f_G(y) = \int ~ \mathcal{K}(y, \theta) dG(\theta) ,
\end{eqnarray}
where $G(\cdot)$ is a mixing distribution. In particular, by choosing
$G$ to be a discrete distribution with finitely many atoms $\theta_h, h = 1, \ldots, k$ having weights $\pi_h, h= 1, \ldots, k$
with $\sum_{h=1}^k \pi_h = 1$, one obtains the important class of finite mixture models. In a Bayesian framework, one can
induce a prior distribution on the class of densities by assigning a prior to $G$, which amounts to specifying priors on $k$ and $(\theta_h, \pi_h), h = 1, \ldots, k$ in case of finite mixture models.
A Dirichlet process \citep{ferguson1973bayesian,ferguson1974prior} is often used as a default prior on the class of mixing distributions due to its attractive theoretical properties and availability of efficient algorithms for posterior computation. Since realizations of a Dirichlet process are almost surely discrete (see \cite{sethuraman1994constructive} for a constructive definition), a Dirichlet process prior on $G$ induces an infinite discrete mixture model for $f_G$. A well known drawback of finite mixture models is the sensitivity of the results to the choice of $k$, whereas updating $k$ in a fully Bayesian formulation is computationally intensive. The infinite mixture representation avoids fixing a truncation level and sophisticated sampling algorithms such as \cite{walker2007sampling} enable posterior sampling from the full posterior distribution.

Although finite and infinite discrete mixture models have been extensively used, there are reasons to look beyond these classes of models. A discrete prior on $G$ partitions the $n$ subjects into one or more clusters, with subjects in the same cluster sharing the same $\theta$ value. Although this property has been widely exploited for probabilistic clustering, one might want to avoid the clustering phenomenon in situations where the interest is purely in density estimation and one is not interested in interpreting the clusters or in inferring the cluster specific parameters. It is often the case that the clusters don't have any physical significance and subjects get inappropriately grouped together for all parameter values obscuring subtle differences. In such cases, the clustering is more of an artifact of the model and a continuum among the parameter values for the subjects seems more reasonable.

While Polya tree priors \citep{ferguson1974prior,mauldin1992polya} can be directly used to induce priors on the space of absolutely continuous densities \citep{lavine1992some}, the resulting density estimates are found to be spiky in practice. \cite{lenk1988logistic,lenk1991towards} proposed a logistic Gaussian process which bypasses the mixture formulation by directly modeling an unknown density on the unit interval as the exponent of a random function re-normalized, or equivalently modeling the log-density using a Gaussian process prior. The normalizing constant in the logistic Gaussian process models is analytically intractable and causes difficulties in posterior sampling. Refer to \cite{tokdar2007towards} for a faster implementation in density estimation with logistic Gaussian process priors.

Recently, \cite{kundu2011bayesian} proposed an approach for univariate density estimation in which the response variables are modeled as unknown functions of uniformly distributed latent variables with an additive Gaussian error. The latent variable specification allows straightforward posterior computation via conjugate posterior updates. Since inverse c.d.f. transforms of uniform random variables can generate draws from any distribution, by choosing the prior on the error variance to assign positive mass to arbitrary neighborhoods of zero while placing a prior with large support on the space of functions mapping the latent variables to the observed variables (referred to as the \emph{transfer function} from now on), their prior can approximate draws from any continuous distribution function arbitrarily closely. One can also conveniently center the non-parametric model on a parametric family by centering the prior on the transfer function on a parametric class of quantile (or inverse c.d.f.) functions $\{F_{\theta}^{-1} ~:~ \theta \in \Theta\}$. While such centering on parametric guesses can be achieved in Dirichlet process mixture models by appropriate choice of the base measure $G_0$, posterior computation becomes complicated unless the base measure is conjugate to the kernel $\mathcal{K}$.

There has been growing interest in studying asymptotic properties of Bayesian procedures assuming the data are sampled from a fixed unknown distribution. The posterior distribution is said to be strongly consistent if it concentrates almost surely in arbitrarily small $L_1$ neighborhoods of the true distribution with increasing sample size. \cite{ghosal1999posterior} provided general conditions in terms of $L_1$ metric entropy to ensure strong posterior consistency and verified those conditions for Dirichlet process location mixtures of normal kernels under certain regularity conditions. \cite{tokdar2006posterior} extended their result to the location-scale mixture case while encompassing a significantly larger class of ``true'' densities. \cite{ghosal2000convergence} considered the rate of contraction of a posterior distribution to the true density, providing an upper bound on the rate at which one can let the neighborhood size decrease to zero. \cite{ghosal2001entropies} obtained rates of posterior contraction for the Dirichlet process mixture model when the true density is a location-scale mixture of normals with component specific standard deviations bounded between two positive numbers. Although a nearly parametric rate is obtained in this case, the above class of densities is restrictive since one needs the component specific standard deviations to be arbitrarily small for normal mixtures to be able to approximate any smooth density. \cite{ghosal2007posterior} developed a generalization of the basic rate theorem in \citet{ghosal2000convergence} and addressed a broader class of densities, namely, the class of twice continuously differentiable densities. Under some regularity conditions which include the requirement that the true density be compactly supported, they obtained the optimal minimax rate of $n^{-2/5}$ up to a logarithmic factor based on Dirichlet process mixture models. \cite{kruijer2010adaptive} considered finite location-scale mixtures of exponential power distributions and obtained minimax rates of convergence up to a logarithmic factor for any $\beta$-H\"{o}lder density, implying rate adaptivity to any degree of smoothness of the true density.

In this article, we study rates of posterior contraction in univariate density estimation for a class of non-linear latent variable models (NL-LVM) similar to \cite{kundu2011bayesian}. The NL-LVM encompasses a large class of univariate densities and it is straightforward to extend the class for multivariate density estimation and density regression problems. In particular, the NL-LVM has elements in common to Gaussian process latent variable models (GP-LVM) routinely used in machine learning applications for high-dimensional data visualization and dimensionality reduction \citep{lawrence2004gaussian,lawrence2005probabilistic,lawrence2007hierarchical,ferris2007wifi}.
However, the literature on GP-LVM doesn't provide any discussion on the flexibility of their specification in terms of the induced density of the observations after marginalizing out the latent variables. Although \citet{kundu2011bayesian} provide an intuitive argument for large support in the density space for the univariate case, a rigorous characterization of the prior support is missing.
We provide an accurate characterization of the prior support in terms of kernel convolution with a class of continuous mixing measures. We provide conditions for the mixing measure to admit a density with respect to Lebesgue measure and show that the prior support of the NL-LVM is at least as large as that of DP mixture models. We then develop approximation results for the above class of continuous mixing measures and subsequently derive posterior contraction rates assuming standard smoothness assumptions on the true density. Assuming the true density to be twice continuously differentiable, the best obtainable rate is found to be the minimax rate of $n^{-2/5}$ up to a logarithmic factor.  Further, if the prior on the transfer function is centered on a parametric family which happens to  contain the true density, then one gets a faster convergence rate which can be arbitrarily close to the parametric rate of $n^{-1/2}$ up to a logarithmic factor. Also, analogous to the Dirichlet process mixture models, when the true density is a Gaussian convolution with a finite mixture of truncated Gaussians, one can also attain a near parametric convergence rate. 

The main contributions of this article are as follows. (i) The characterization of our model using convolutions implies that one can approximate any continuous density by choosing the transfer function to be the quantile function of the true density and letting the error variance to decrease to zero. When the true density is not compactly supported, the corresponding quantile function is unbounded with discontinuities at $0$ and $1$ and it is not immediate whether a prior for the transfer function supported on $C[0,1]$ (a default choice being a Gaussian process prior) results in the optimal rate. To address this issue, we define a sequence of $C[0,1]$ functions that converge pointwise to the true quantile function and derive concentration bounds for the prior around this sequence. (ii) The traditional approach of approximating the Gaussian convolution of a compactly supported density by discrete normal mixtures isn't well-suited for our purpose since the quantile function of the mixing distribution is a step function which doesn't belong to the sup-norm support of any smooth stochastic process. We develop a technique based on maximum entropy moment matching \citep{mead1984maximum} for approximating a compactly supported density by an infinitely smooth density. Although the above developments are crucially used for our treatment of the non-compact case, we believe these results will be of independent interest.

The rest of the article is organized as follows. We introduce relevant notations and terminologies in Section \ref{sec:notn}. To make the article self-contained, we also provide a brief background on Gaussian process priors. In Section \ref{sec:assf0}, we formulate our assumptions on the true density $f_0$ and in the following section, we describe the NL-LVM model and relate it to convolutions. We state our main theorem on convergence rates for the compact case in Section \ref{sec:compact} and  the non-compact case in Section \ref{sec:noncompact}. We discuss some special cases in Section \ref{sec:splcases}. Section \ref{sec:discussion} discusses some implications of our results and outlines possible future directions.

\section{Notations}\label{sec:notn}

Throughout the article, $Y_1,\ldots, Y_n, \ldots$ are independent and identically
distributed with density $f_0 \in \mathcal{F}$, the set of all densities on $\mathbb{R}$ absolutely continuous with respect to the Lebesgue measure $\lambda$. The supremum and $\mbox{L}_1$-norm are denoted by $\norm{\cdot}_{\infty}$ and $\norm{\cdot}_{1}$, respectively. We let $\norm{\cdot}_{p, \nu}$ denote the norm of
$L_p(\nu)$, the space of measurable functions with $\nu$-integrable $p$th absolute power. For two density functions $f, g \in \mathcal{F}$, let $h$ denote the Hellinger distance defined as $h^2(f, g) = \norm{\sqrt{f}-\sqrt{g}}_{2, \lambda}=\int (f^{1/2} - g^{1/2})^2 d\lambda $, $K(f,g)$ the Kullback-Leibler divergence given by $K(f,g) = \int \log(f/g) f d\lambda$ and $V(f,g)=\int \log(f/g)^2 f d\lambda$. The notation $C[0, 1]$ is used for the space of continuous functions $f : [0, 1] \rightarrow \mathbb{R}$ endowed with the supremum norm. For $\beta >0$, we let $C^{\beta}[0, 1]$ denote the H\"{o}lder space of order $\beta$, consisting of the functions $f \in C[0, 1]$ that have $\lfloor \beta \rfloor$ continuous
derivatives  with the $\lfloor \beta \rfloor$th derivative
$f^{\lfloor \beta \rfloor}$ being Lipschitz continuous of order $\beta -\lfloor \beta \rfloor$. The $\epsilon$-covering number $N(\epsilon,S,d)$ of a semi-metric space $S$ relative to the semi-metric $d$ is the minimal number of balls of radius $\epsilon$ needed to cover $S$. The logarithm of the
covering number is referred to as the entropy.  By near-optimal rate of convergence we mean optimal rate of convergence slowed down by a logarithmic factor.

We write ``$\precsim$'' for inequality up to a constant multiple. Let \linebreak $\phi(x) = (2\pi)^{-1/2}\exp(-x^2/2)$ denote the standard normal density, and let $\phi_{\sigma}(x) = (1/\sigma) \phi(x/\sigma)$. Let an asterisk denote a convolution e.g., $(\phi_{\sigma} * f)(y) = \int \phi_{\sigma}(y - x)f(x)dx$.
The support of a density $f$ is denoted by supp($f$).

We briefly recall the definition of the RKHS of a Gaussian process prior; a detailed review can be found in \cite{van2008reproducing}. A Borel measurable random element $W$ with values in a separable Banach space $(\mathbb{B}, \norm{\cdot})$ (e.g., $C[0,1]$) is called Gaussian if the random variable $b^*W$ is normally distributed for any element $b^* \in \mathbb{B}^*$, the dual space of $\mathbb{B}$. The reproducing kernel
Hilbert space (RKHS) $\mathbb{H}$ attached to a zero-mean Gaussian process $W$ is
defined as the completion of the linear space of functions $t \mapsto EW(t)H$
relative to the inner product
\begin{eqnarray*}
\langle \mbox{E} W(\cdot)H_1; \mbox{E} W(\cdot )H_2\rangle_{\mathbb{H}} = \mbox{E} H_1H_2,
\end{eqnarray*}
where $H, H_1$ and $H_2$ are finite linear combinations of the form $\sum_{i}a_{i}W(s_i)$
with $a_i \in \mathbb{R}$ and $s_i$ in the index set of $W$.

\section{Assumptions on the true density}\label{sec:assf0}

It has been widely recognized that one needs certain smoothness assumptions and tail conditions on the true density $f_0$ to derive posterior convergence rates at $f_0$. We need the following assumptions in our case,
\begin{ass}\label{ass:1}
$f_0$ is twice continuously differentiable with $\int (f_0^{\prime\prime}/f_0)^2f_0d\lambda  < \infty$ and
$\int(f_0^{\prime}/f_0)^4f_0d\lambda < \infty$.
\end{ass}
\begin{rem}\label{rem:ass:1}
Letting $f_0(y) = C\exp\{ - w_0(y)\}$ on $\mbox{supp}(f_0)$ so that $w_0 = \log C - \log f_0(y)$, we can restate Assumption \ref{ass:1}  as  $w_0$ being twice continuously differentiable and
\begin{eqnarray}\label{eq:tails}
\int_{-\infty}^{\infty}\{w_{0}^{\prime}(y)\}^4\exp\{ -w_0(y)\} < \infty, \, \int_{-\infty}^{\infty}\{w_{0}^{\prime\prime}(y)\}^2\exp\{ -w_0(y)\} < \infty.
\end{eqnarray}
\end{rem}

\begin{ass}\label{ass:2}
$f_0$ is bounded, nondecreasing on $(-\infty, a]$, bounded away from $0$ on $[a,b]$ and non-increasing on $[b, \infty)$ for some $a \leq b$.
\end{ass}
Assumption \ref{ass:1} is the same as Assumption 1.2 of \cite{ghosal2007posterior} and ensures that $h(f_0, f_0*\phi_{\sigma}) = O(\sigma^2)$ as $\sigma \rightarrow 0$; see Lemma 4 of \cite{ghosal2007posterior} for a proof. Assumption \ref{ass:2} is the same as the assumption in Lemma 6 of the same paper. This is sufficient to guarantee that for every $\delta > 0$, there exists a constant $C > 0$ such that $f_0 * \phi_{\sigma} \geq Cf_0$ for every $\sigma < \delta$.
While Assumption \ref{ass:1} only allows sufficiently smooth densities, Assumption \ref{ass:2} is only a mild requirement in the sense that most reasonable densities arising in practice should satisfy it. Moreover, if $f_0$ is nondecreasing on $(-\infty, a]$ and nonincreasing on $[b, \infty]$ for some $a \leq b$, $f_0$ is automatically bounded and bounded away from zero on $[a, b]$ provided it is continuous and no-where zero on $[a, b]$.

\section{The NL-LVM model}\label{sec:model}

Consider the nonlinear latent variable model,
\begin{align}\label{eq:nl-lvm}
y_i &= \mu(\eta_i) + \epsilon _i, \quad \epsilon_i \sim \mbox{N}(0, \sigma^2), \, (i=1, \ldots, n)\\
\mu &\sim \Pi_{\mu}, \quad \sigma \sim \Pi_{\sigma},\quad \eta_i \sim \mbox{U}(0,1),
\end{align}
where $\eta_i$'s are subject specific latent variables, $\mu \in C[0, 1]$ is a \emph{transfer function} relating the latent variables to the observed variables and $\epsilon_i$ is an idiosyncratic error specific to subject $i$. The density of $y$ conditional on the transfer function $\mu$ and scale $\sigma$ is obtained on marginalizing out the latent variable as
\begin{eqnarray}\label{eq:gpt_def}
f(y; \mu, \sigma) \stackrel{\text{def}}{=} f_{\mu, \sigma}(y)= \int_{0}^{1}\phi_{\sigma}(y-\mu(x))dx.
\end{eqnarray}
Define a map $g: C[0,1] \times [0,\infty) \to \mathcal{F}$ with $g(\mu, \sigma) = f_{\mu, \sigma}$.
One can induce a prior $\Pi$ on $\mathcal{F}$ via the mapping $g$ by placing independent priors
$\Pi_{\mu}$ and $\Pi_{\sigma}$ on $C[0,1]$ and $[0, \infty)$ respectively, with $\Pi = (\Pi_{\mu} \otimes \Pi_{\sigma}) \circ g^{-1}$. \cite{kundu2011bayesian} assumed a Gaussian process prior with squared exponential covariance kernel on $\mu$ and an inverse-gamma prior on $\sigma^2$.

It is not immediately clear whether the class of densities $f_{\mu, \sigma}$ in the range of $g$ encompass a large subset of the density space. We  provide an intuition that relates the above class with convolutions and is crucially used later on. Let $f_0$ be a continuous density with cumulative distribution function $F_0(t) = \int_{-\infty}^{t} f_0(x) dx$. Assume $f_0$ to be non-zero almost everywhere within its support, so that $F_0 : \mbox{supp}(f_0) \to [0,1]$ is strictly monotone and hence has an inverse $F_0^{-1} : [0,1] \to \mbox{supp}(f_0)$ satisfying $F_0\{F_0^{-1}(t)\} = t$ for all $t \in \mbox{supp}(f_0)$. If $\mbox{supp}(f_0) = \mathbb{R}$, then the domain of $F_0^{-1}$ is the open interval $(0,1)$ instead of $[0,1]$.

Letting $\mu_0(x) = F_0^{-1}(x)$, one obtains
\begin{eqnarray}\label{eq:conv}
f_{\mu_0, \sigma}(y) = \int_{0}^{1} \phi_{\sigma}(y-F_0^{-1}(x))dx = \int_{-\infty}^{\infty} \phi_{\sigma}(y-t) f_0(t) dt,
\end{eqnarray}
where the second equality follows from the change of variable theorem. Thus, $f_{\mu_0,\sigma}(y) = \phi_{\sigma}*f_0$, i.e., $f_{\mu_0, \sigma}$ is the convolution of $f_0$ with a normal density having mean $0$ and standard deviation $\sigma$. It is well known that the convolution $\phi_{\sigma}*f_0$ can approximate $f_0$ arbitrary closely as the bandwidth $\sigma \to 0$. More precisely, for $f_0 \in L^p(\lambda)$ for any $p \geq1$, $\norm{\phi_{\sigma}*f_0 - f_0}_{p, \lambda} \to 0$ as $\sigma \to 0$. Furthermore, a stronger result $\norm{\phi_{\sigma}*f_0 - f_0}_{\infty} = O(\sigma^2)$ holds if $f_0$ is compactly supported.
A similar result holds for the Hellinger metric, with the precise approximation error under Assumption \ref{ass:1} given by $h(\phi_{\sigma}*f_0, f_0) = O(\sigma^2)$ as $\sigma \to 0$.

Suppose the prior $\Pi_{\mu}$ on $\mu$ has full sup-norm support on $C[0,1]$ so that $\mbox{Pr}(\norm{\mu - \mu^*}_{\infty} < \epsilon) > 0$ for any $\epsilon > 0$ and $\mu^* \in C[0,1]$, and the prior $\Pi_{\sigma}$ on $\sigma$ has full support on $[0, \infty)$. If $f_0$ is compactly supported so that the quantile function $\mu_0 \in C[0,1]$, then it can be shown that under mild conditions, the induced prior $\Pi$ assigns positive mass to arbitrarily small $L_1$ neighborhoods of any density $f_0$. When $f_0$ has full support on $\mathbb{R}$, the quantile function $\mu_0$ is unbounded near $0$ and $1$, so that $\norm{\mu_0}_{\infty} = \infty$. However, $\int_{0}^{1} \abs{\mu_0(t)} dt = \int_{\mathbb{R}} \abs{x} f_0(x) dx$, which implies that $\mu_0$ can be identified as an element of $L_1[0,1]$ if $f_0$ has finite first moment. Since $C[0,1]$ is dense in $L_1[0,1]$, the previous conclusion regarding $L_1$ support can be shown to hold in the non-compact case too. We summarize the above discussion in the following theorem, with a proof provided in the appendix.
\begin{thm}\label{thm:support}
If $\Pi_{\mu}$ has full sup-norm support on $C[0,1]$ and $\Pi_{\sigma}$ has full support on $[0, \infty)$, then the $L_1$ support of the induced prior $\Pi$ on $\mathcal{F}$ contains all densities $f_0$ which have a finite first moment and are non-zero almost everywhere on their support.
\end{thm}
\begin{rem}
The conditions of Theorem \ref{thm:support} are satisfied for a wide range of Gaussian process priors on $\mu$ (for example, a GP with a squared exponential or Mat\'{e}rn covariance kernel).
\end{rem}

Let $\tilde{\lambda}$ denote the Lebesgue measure on $[0,1]$, or equivalently, the $\mbox{U}[0,1]$ distribution. For any measurable function $\mu : [0,1] \to \mathbb{R}$, let $\nu_{\mu}$ denote the induced measure on $(\mathbb{R}, \mathcal{B})$, with $\mathcal{B}$ denoting the Borel sigma-field on $\mathbb{R}$. Then, for any Borel measurable set $B$, $\nu_{\mu}(B) = \tilde{\lambda}(\mu^{-1}(B))$, where $\mu^{-1}(B) = \{x \in [0,1] ~ : ~ \mu(x) \in B\}$. By the change of variable theorem for induced measures,
\begin{eqnarray}\label{eq:dpgp}
\int_{0}^{1}\phi_{\sigma}(y-\mu(x))dx = \int \phi_{\sigma}(y-t) d\nu_{\mu}(t),
\end{eqnarray}
so that $f_{\mu, \sigma}$ can be expressed as a kernel mixture form as in (\ref{eq:kernel}) with mixing distribution $\nu_{\mu}$. It turns out that this mechanism of creating random distributions is very general. Depending on the choice of $\mu$, one can create a large variety of mixing distributions based on this specification. For example, if $\mu$ is a strictly monotone function, then $\nu_{\mu}$ is absolutely continuous with respect to the Lebesgue measure, while choosing $\mu$ to be a step function, one obtains a discrete mixing distribution. However, it is easier to place a prior on $\mu$ supported on the space of continuous functions $C[0, 1]$ without further shape restrictions and Theorem \ref{thm:support} assures us that this specification leads to large $L_1$ support on the space of densities.

\section{The compact case}\label{sec:compact}

We first consider the case where $f_0$ is compactly supported, i.e., there exist $-\infty < a_0 < b_0 < \infty$ such that $\int_{a_0}^{b_0}f_0(x)=1$. In that case, the quantile function $F_0^{-1}: [0,1] \to [a_0, b_0]$ is a continuous monotone function inheriting the smoothness of $f_0$. Denote the quantile function by $\mu_0$. Assumption \ref{ass:1} ensures that the compactly supported density decays smoothly at the boundaries. Under Assumption \ref{ass:1} and the fundamental theorem of calculus, $\mu_0:[0,1] \to [a_0,b_0]$ is thrice continuously differentiable implying $\mu_0 \in C^{3}[0,1]$.

\subsection{Prior specification}

We now mention our choices for the prior distributions $\Pi_{\mu}$ and $\Pi_{\sigma}$.
\begin{ass}\label{ass:1p}
We assume $\mu$ follows a centered Gaussian process denoted by $\mbox{GP}(0, c)$, with a squared exponential covariance kernel $c(\cdot, \cdot; A)$ and a Gamma prior for the inverse-bandwidth $A$. Thus
$c(t, s; A) = e^{-A(t-s)^2}, t,s \in [0, 1], A \sim \mbox{Ga}(p,q)$.
\end{ass}
\begin{ass}\label{ass:2p}
We assume $\sigma \sim \mbox{IG}(a_{\sigma}, b_{\sigma})$.
\end{ass}
Note that contrary to the usual conjugate choice of an inverse-Gamma prior for $\sigma^2$, we have assumed an inverse-Gamma prior for $\sigma$.  This enables one to have slightly more prior mass near zero compared to an inverse-Gamma prior for $\sigma^2$, leading to the optimal rate of posterior convergence. Refer also to \cite{kruijer2010adaptive} for a similar prior choice for the bandwidth of the kernel in discrete location-scale mixture priors for densities.

\subsection{Posterior convergence rate for the compact case}

We state below the main theorem of posterior convergence rates.
 \begin{thm}\label{thm:compact}
If $f_0$ satisfies Assumption \ref{ass:1} and the priors $\Pi_\mu$ and $\Pi_{\sigma}$ are as in
Assumptions \ref{ass:1p} and \ref{ass:2p} respectively, the best obtainable rate of posterior convergence relative to $h$ is
\begin{eqnarray}\label{eq:optrate}
\epsilon_{n} = n^{-\frac{2}{5}}\log n.
\end{eqnarray}
\end{thm}
The proof of Theorem \ref{thm:compact} is based on \citet{ghosal2007posterior} and \linebreak \citet{van2007bayesian,van2008reproducing,van2009adaptive}. Unlike the treatment in discrete mixture models \citep{ghosal2007posterior} where a compactly supported density is approximated with a discrete mixture of normals, the main trick here is to approximate the true density $f_0$ by the convolution $\phi_{\sigma}*f_0$ and allow the prior on the transfer function to appropriately concentrate around the true quantile function $\mu_0 \in C[0,1]$.

To guarantee that the above scheme leads to the optimal rate of convergence, we first derive sharp bounds for the Hellinger distance between $f_{\mu_1, \sigma_1}$ and $f_{\mu_2, \sigma_2}$ for $\mu_1, \mu_2 \in C[0, 1]$ and $\sigma_1, \sigma_2 > 0$. We summarize the result in the following Lemma \ref{lem:hellinger}.
\begin{lem}\label{lem:hellinger}
For $\mu_1, \mu_2 \in C[0, 1]$ and $\sigma_1, \sigma_2 > 0$,
\begin{eqnarray}
h^2(f_{\mu_1,\sigma_1}, f_{\mu_2, \sigma_2}) \leq 1- \sqrt{\frac{2\sigma_1\sigma_2}{\sigma_1^2 + \sigma_2^2}}\exp\bigg\{-\frac{\norm{\mu_1 - \mu_2}_{\infty}^2}{4(\sigma_1^2 + \sigma_2^2)}\bigg\}.
\end{eqnarray}
\end{lem}
\begin{proof}
Note that by H\"{o}lder's inequality,
\begin{eqnarray*}
f_{\mu_1, \sigma_1}(y)f_{\mu_2, \sigma_2}(y) \geq \bigg\{\int_{0}^{1} \sqrt{\phi_{\sigma_1}(y - \mu_1(x))} \sqrt{\phi_{\sigma_2}(y - \mu_2(x))}dx\bigg\}^2.
\end{eqnarray*}
Hence,
\begin{eqnarray*}
h^2(f_{\mu_1,\sigma_1}, f_{\mu_2, \sigma_2}) &\leq& \int\bigg[\int_{0}^{1}\phi_{\sigma_1}(y-\mu_1(x))dx  + \int_{0}^{1}\phi_{\sigma_2}(y-\mu_2(x))dx \\
&-&2\int_{0}^{1}\sqrt{\phi_{\sigma_1}(y - \mu_1(x))} \sqrt{\phi_{\sigma_2}(y - \mu_2(x))}dx\bigg]dy.
\end{eqnarray*}
By changing the order of integration (applying Fubini's theorem since the function within the integral is jointly integrable) we get,
\begin{eqnarray*}
h^2(f_{\mu_1,\sigma_1}, f_{\mu_2, \sigma_2}) &\leq& \int_{0}^{1}h^2(f_{\mu_1(x),\sigma_1}, f_{\mu_2(x), \sigma_2})dx \\
&=& \int_{0}^{1}\bigg[1- \sqrt{\frac{2\sigma_1\sigma_2}{\sigma_1^2 + \sigma_2^2}}\exp\bigg\{-\frac{(\mu_1(x) - \mu_2(x))^2}{4(\sigma_1^2 + \sigma_2^2)}\bigg\}\bigg]dx\\
&\leq& 1- \sqrt{\frac{2\sigma_1\sigma_2}{\sigma_1^2 + \sigma_2^2}}\exp\bigg\{-\frac{\norm{\mu_1 - \mu_2}_{\infty}^2}{4(\sigma_1^2 + \sigma_2^2)}\bigg\}.
\end{eqnarray*}

\end{proof}

\begin{rem}
When $\sigma_1 = \sigma_2  = \sigma$, $h^2(f_{\mu_1,\sigma}, f_{\mu_2, \sigma}) \leq 1 - \exp\big\{\norm{\mu_1 -\mu_2}_{\infty}^2/ 8 \sigma^2\big\}$, which implies that $h^2(f_{\mu_1,\sigma}, f_{\mu_2, \sigma}) \precsim \norm{\mu_1 -\mu_2}_{\infty}^2/\sigma^2$.
\end{rem}

\begin{rem}
Note that if we had used $h^2(f_{\mu_1,\sigma_1}, f_{\mu_2, \sigma_2}) \leq \norm{f_{\mu_1,\sigma_1}- f_{\mu_2, \sigma_2}}_1$, we would have obtained the cruder bound
\begin{eqnarray*}
h^2(f_{\mu_1,\sigma_1}, f_{\mu_2, \sigma_2}) \leq C_1 \frac{\norm{\mu_1 -\mu_2}_{\infty}}{(\sigma_1 \wedge \sigma_2)} + C_2\frac{|\sigma_2 - \sigma_1|}{(\sigma_1 \wedge \sigma_2)},
 \end{eqnarray*}
which is linear in $\norm{\mu_1 -\mu_2}_{\infty}$ for some constant $C_1, C_2 > 0$. This bound is less sharp than what is obtained in Lemma \ref{lem:hellinger} and does not suffice for obtaining the optimal rate of convergence.
\end{rem}

To control the Kullback-Leibler distance between the true density $f_0$ and the model $f_{\mu, \sigma}$, we derive an upper bound for $\log \norm{\frac{f_0}{f_{\mu, \sigma}}}_{\infty}$ in Lemma \ref{lem:logsup}.
\begin{lem}\label{lem:logsup}
If $f_0$ satisfies Assumption \ref{ass:2},
\begin{eqnarray}
\log \norm{\frac{f_0}{f_{\mu, \sigma}}}_{\infty} \leq C_6 + \frac{\norm{\mu - \mu_0}_{\infty}^2}{\sigma^2}
\end{eqnarray}
for some constant $C_6 > 0$.
\end{lem}
\begin{proof}
Note that
\begin{eqnarray*}
f_{\mu, \sigma}(y) &=& \frac{1}{\sqrt{2\pi}\sigma}\int_{0}^{1}\exp\bigg\{-\frac{(y-\mu(x))^2}{2\sigma^2}\bigg\}dx\\
&\geq&  \frac{1}{\sqrt{2\pi}\sigma}\int_{0}^{1}\exp\bigg\{-\frac{(y-\mu(x))^2}{\sigma^2}\bigg\}dx \exp\bigg\{-\frac{\norm{\mu-\mu_0}_{\infty}^2}{\sigma^2}\bigg\}\\
&\geq& C_4\phi_{\sigma/\sqrt{2}} * f_0 (y) \exp\bigg\{-\frac{\norm{\mu-\mu_0}_{\infty}^2}{\sigma^2}\bigg\}\\
&\geq& C_5 f_0(y) \exp\bigg\{-\frac{\norm{\mu-\mu_0}_{\infty}^2}{\sigma^2}\bigg\},
\end{eqnarray*}
where the last inequality follows from Lemma 6 of \cite{ghosal2007posterior} by Assumption \ref{ass:2}.
Hence
$\log \norm{\frac{f_0}{f_{\mu, \sigma}}}_{\infty} \leq C_6 + \frac{\norm{\mu - \mu_0}_{\infty}^2}{\sigma^2}$ for some constant $C_6 > 0$.
\end{proof}

\begin{rem}
Note that if $f_0$ is compact then Assumption \ref{ass:2} is automatically satisfied.
\end{rem}

\noindent \textbf{Proof of Theorem \ref{thm:compact}:}
Following \cite{ghosal2000convergence}, we need to find sequences $\bar{\epsilon}_n,\tilde{\epsilon}_n \to 0$ with
$n\min\{\bar{\epsilon}_n^2,\tilde{\epsilon}_n^2\} \to \infty$  such that there exist constants $C_1, C_2, C_3, C_4> 0$ and sets $\mathcal{F}_n \subset \mathcal{F}$ so that,
\begin{align}
& \log N(\epsilon_n, \mathcal{F}_n, d) \leq C_1n\bar{\epsilon}_n^2 \label{eq1}\\
& \Pi(\mathcal{F}_n^c) \leq C_3\exp\{-n\tilde{\epsilon}_n^2(C_2+4)\}\label{eq2} \\
& \Pi\bigg( f_{\mu, \sigma}:  \int f_0 \log \frac{f_0}{f_{\mu, \sigma}} \leq \tilde{\epsilon}_n^2, \int f_0 \log \bigg(\frac{f_0}{f_{\mu, \sigma}}\bigg)^2 \leq \tilde{\epsilon}_n^2 \bigg) \geq C_4\exp\{-C_2n\tilde{\epsilon}_n^2\} \label{eq3}.
\end{align}
Then we can conclude that for $\epsilon_n = \max\{\bar{\epsilon}_n, \tilde{\epsilon}_n\}$ and sufficiently large $M > 0$, the posterior probability
\begin{eqnarray*}
\Pi_n(f_{\mu, \sigma}: d(f_{\mu, \sigma}, f_0) > M\epsilon_n | Y_1, \ldots, Y_n) \to 0 \, \, \text{a.s.}\, P_{f_0}.
\end{eqnarray*}

Let $W = (W_t: t \in \mathbb{R})$ be a Gaussian process with squared exponential covariance kernel. The spectral measure $m_{w}$ of $W$ is absolutely continuous with respect to the Lebesgue measure $\lambda$ on $\mathbb{R}$ with the Radon-Nikodym derivative given by
\begin{eqnarray*}
\frac{dm_{w}}{d\lambda}(x) = \frac{1}{2\pi^{1/2}}e^{-x^2/4}.
\end{eqnarray*}

Define a scaled Gaussian process $W^a=(W_{at}: t \in [0,1])$, viewed as a map in $C[0,1]$. Let $\mathbb{H}^a$ denote the RKHS of $W^a$, with the corresponding norm $\norm{\cdot}_{\mathbb{H}^a}$. The unit ball in the RKHS is denoted $\mathbb{H}^a_1$. We will consider the Gaussian process $\mu \sim W^A$ given $A$, with $A \sim \mbox{Gamma}(p,q)$.

We will first verify (\ref{eq3}) along the lines of \cite{ghosal2007posterior}.
Note that
\begin{eqnarray}\label{eq:H2}
h^{2}(f_0, f_{\mu, \sigma}) \precsim h^{2}(f_0, f_{\mu_0, \sigma}) + h^{2}(f_{\mu_0, \sigma}, f_{\mu, \sigma}).
\end{eqnarray}
Since $f_{\mu_0, \sigma} = \phi_{\sigma}*f_0$, using Lemma 4 of \cite{ghosal2007posterior}, one obtains under Assumptions \ref{ass:1} and \ref{ass:2}, \begin{eqnarray}\label{eq:H}
 h^{2}(f_0, f_{\mu_0, \sigma}) \precsim O(\sigma^4).
\end{eqnarray}

From Lemma \ref{lem:hellinger} and the following remark, we obtain
\begin{eqnarray}
h^{2}(f_{\mu_0, \sigma}, f_{\mu, \sigma}) \precsim \frac{\norm{\mu- \mu_0}_{\infty}^2}{\sigma^2}.
\end{eqnarray}
From Lemma 8 of \cite{ghosal2007posterior}, one has
\begin{eqnarray}\label{eq:KtoH}
\int f_0 \log \bigg(\frac{f_0}{f_{\mu, \sigma}}\bigg)^{i} \leq h^2(f_0, f_{\mu,\sigma})\bigg(1 + \log \norm{\frac{f_0}{f_{\mu, \sigma}}}_{\infty}\bigg)^{i}
\end{eqnarray}
for $i=1,2$.

From (\ref{eq:H2})-(\ref{eq:KtoH}), for any $b \geq 1$ and $\tilde{\epsilon}_n^2 = \sigma_n^4$,
\begin{eqnarray*}
\big\{ \sigma \in [\sigma_n, \sigma_n + \sigma_n^b], \norm{\mu - \mu_0}_{\infty} \precsim \sigma_n^3 \big\} \subset \\ \bigg\{  \int f_0 \log \frac{f_0}{f_{\mu, \sigma}} \precsim \sigma_n^4, \int f_0 \log \bigg(\frac{f_0}{f_{\mu, \sigma}}\bigg)^2 \precsim \sigma_n^4\bigg\}.
\end{eqnarray*}
Then (\ref{eq3}) will be satisfied with $\tilde{\epsilon}_n = n^{-\frac{2}{5}}$
if
\begin{eqnarray*}
\mbox{P}\{\sigma \in [\sigma_n, 2\sigma_n ], \norm{\mu - \mu_0}_{\infty} \precsim \sigma_n^3\} \geq \exp\{-C_4n^{\frac{1}{5}}\}
\end{eqnarray*}
for some constant $C_4 > 0$.

Since $\mu_0 \in C^{3}[0,1]$, from Section 5.1 of \cite{van2009adaptive},
\begin{eqnarray*}
\mbox{P}(\norm{\mu - \mu_0}_{\infty} \leq 2\delta_n) \geq C_5\exp\{-C_6(1/\delta_n)^{1/3}\}(C_7/\delta_n)^{p/3},
\end{eqnarray*}
for $\delta_n \to 0$ and constants $C_5, C_6, C_7 > 0$.  Letting $\delta_n = \sigma_n^3$, we obtain
\begin{eqnarray*}
\mbox{P}(\norm{\mu - \mu_0}_{\infty} \leq 2\delta_n) \geq \exp\{-C_8(1/\sigma_n)\},
\end{eqnarray*}
for some constant $C_8 > 0$.
Since $\sigma \sim IG(a_{\sigma}, b_{\sigma})$,
we have
\begin{eqnarray*}
P(\sigma \in [\sigma_n, 2\sigma_n ]) &=& \frac{b_{\sigma}^{a_{\sigma}}}{\Gamma(a_{\sigma})}\int_{\sigma_n}^{2\sigma_n}x^{-(a_{\sigma}+1)} e^{-b_{\sigma}/x}dx \\ &\geq& \frac{b_{\sigma}^{a_{\sigma}}}{\Gamma(a_{\sigma})}\int_{\sigma_n}^{2\sigma_n} e^{-2b_{\sigma}/x}dx \\
&\geq& \frac{b_{\sigma}^{a_{\sigma}}}{\Gamma(a_{\sigma})}\sigma_n\exp\{-b_{\sigma}/\sigma_n \}\\
&\geq& \exp\{-C_9/\sigma_n \},
\end{eqnarray*}
for some constant $C_9> 0$. Hence
\begin{eqnarray*}
\mbox{P}\{\sigma \in [\sigma_n, 2\sigma_n ], \norm{\mu - \mu_0}_{\infty} \precsim \sigma_n^3\} \geq \exp\{-C_4n^{\frac{1}{5}}\},
\end{eqnarray*}
with $\sigma_n = n^{-\frac{1}{5}}$, $\tilde{\epsilon}_n = n^{-\frac{2}{5}}$ and for some $C_4 > 0$.

Next we construct a sequence of subsets $\mathcal{F}_n$ such that \ref{eq1} and \ref{eq2} are satisfied with
$\bar{\epsilon}_n = n^{-\frac{2}{5}}\log ^{t_2}n$ and $\tilde{\epsilon}_n$ for some global constant $t_2 > 0$.

Letting $\mathbb{B}_1$ denote the unit ball of
$C[0,1]$ and given positive sequences $M_n, r_n, \xi_n$, define
\begin{eqnarray*}
B_n = \bigg(M_n\sqrt{\frac{r_n}{\xi_n}}\mathbb{H}_1^r + \bar{\delta}_n\mathbb{B}_1\bigg) \cup
\bigg(\cup_{a < \xi_n}(M_n \mathbb{H}^a_1) + \bar{\delta}_n\mathbb{B}_1 \bigg)
\end{eqnarray*}
as in \cite{van2009adaptive}, with $\bar{\delta}_n =\bar{\epsilon}_nl_n/K_1, K_1 = 2(2/\pi)^{1/2}$ and let
\begin{eqnarray*}
\mathcal{F}_n = \{f_{\mu, \sigma}: \mu \in B_n, l_n < \sigma < h_n \}.
\end{eqnarray*}
First we need to calculate $N(\bar{\epsilon}_n, \mathcal{F}_n, \norm{\cdot}_1)$. Observe that for $\sigma_2 >\sigma_1 > \frac{\sigma_2}{2}$,
\begin{eqnarray*}
\norm{f_{\mu_1, \sigma_1} - f_{\mu_2, \sigma_2}}_1 \leq \bigg(\frac{2}{\pi}\bigg)^{1/2}\frac{\norm{\mu_1 - \mu_2}_{\infty}}{\sigma_1} + \frac{3(\sigma_2 - \sigma_1)}{\sigma_1}.
\end{eqnarray*}

Taking $\kappa_n =\min\{\frac{\bar{\epsilon}_n}{6}, 1\}$ and $\sigma_m^n = l_n (1+ \kappa_n)^m, m \geq 0$, we obtain a partition of $[l_n, h_n]$ as $l_n=\sigma_0^n < \sigma_1^n < \cdots < \sigma_{m_{n}-1}^n < h_n \leq \sigma_{m_n}^n$ with
\begin{eqnarray}\label{eq:entropy1}
m_n= \bigg(\log \frac{h_n}{l_n}\bigg) \frac{1}{\log( 1+\kappa_n)} +1.
\end{eqnarray}
One can show that $\frac{3(\sigma_m^n -\sigma_{m-1}^n)}{\sigma_{m-1}^n} = 3\kappa_n \leq \bar{\epsilon}_n/2$. Let
$\{\tilde{\mu}_k^n, k=1, \ldots, N(\bar{\delta}_n, B_n, \norm{\cdot}_{\infty})\}$  be
a $\bar{\delta}_n$-net of $B_n$. Now consider the set
\begin{eqnarray} \label{eq:net}
\{(\tilde{\mu}_k^n, \sigma_{m}^n): k=1, \ldots, N(\bar{\delta}_n, B_n, \norm{\cdot}_{\infty}), 0\leq m\leq m_n \}.
\end{eqnarray}
Then for any $f= f_{\mu, \sigma} \in \mathcal{F}_n$, we can find $(\tilde{\mu}_k^n, \sigma_{m}^n)$ such that
$\norm{\mu - \tilde{\mu}_k^n}_{\infty} < \bar{\delta}_n$. In addition, if one has $\sigma \in (\sigma_{m-1}^{n}, \sigma_m^{n}]$,
then
\begin{eqnarray*}
\norm{f_{\mu, \sigma} - f_{\mu_k^n, \sigma^n_m}}_1 \leq \bar{\epsilon}_n.
\end{eqnarray*}
Hence the set in (\ref{eq:net}) is an $\bar{\epsilon}_n$-net of $\mathcal{F}_n$ and its covering number is given by
\begin{eqnarray*}
m_nN(\bar{\delta}_n, B_n, \norm{\cdot}_{\infty}).
\end{eqnarray*}
From the proof of Theorem 3.1 in \cite{van2009adaptive}, for $\xi_n = \bar{\delta}_n/ (2\tau M_n)$
and for any $M_n, r_n$ with
\begin{eqnarray}\label{eq:condentropy}
M_n^{3/2}\sqrt{2\tau r_n} > 2\bar{\delta}_n^{3/2}, \, r_n > a_0, \,  M_n \sqrt{\norm{m_{w}}} > \bar{\delta}_n,
\end{eqnarray}
we obtain
\begin{eqnarray}\label{eq:entropy2}
\log N(3\bar{\delta}_n, B_n, \norm{\cdot}_{\infty}) \leq K_2 r_n \bigg( \log \bigg(\frac{M_n^{3/2}\sqrt{2\tau r_n}}{\bar{\delta}_n^{3/2}}\bigg)\bigg)^{2} + 2 \log \frac{2M_n\norm{m_{w}}}{\bar{\delta}_n}
\end{eqnarray}
where $\tau^2 = \int_{\mathbb{R}}x^2 dm_{w}(x)$ and $\norm{m_{w}}$ is the total variation norm of the spectral measure $m_{w}$.

Again from the proof of Theorem 3.1 in \cite{van2009adaptive}, for sufficiently small $\bar{\delta}_n$ and $r_n > 1$ and for $M_n^2 > 16K_3r_n (\log (r_n / \bar{\delta}_n))^2$, we have
\begin{eqnarray}\label{eq:compact1}
P(W^A \notin B_n) \leq K_4r_n^{p-1}\exp\{-K_5r_n\} + \exp\{-M_n^2/8\}
\end{eqnarray}
for constants $K_3, K_4, K_5 > 0$.

Next we calculate $P(\sigma \notin [l_n, h_n])$. Observe that
\begin{eqnarray}
P(\sigma \notin [l_n, h_n ]) &=& P(\sigma^{-1} < h_n^{-1}) + P(\sigma^{-1} > l_n^{-1})\nonumber\\
 &\leq& \sum_{k=\alpha_{\sigma}}^{\infty}\frac{e^{-b_{\sigma}h_n^{-1}}(b_{\sigma}h_n^{-1})^k}{k!} + \frac{b_{\sigma}^{a_{\sigma}}}{\Gamma(a_{\sigma})}\int_{l_n^{-1}}^{\infty} e^{-b_{\sigma}x/2}dx \nonumber \\
&\leq& e^{-a_{\sigma}\log(h_n)} + \frac{b_{\sigma}^{a_{\sigma}}}{\Gamma(a_{\sigma})}e^{-b_{\sigma}l_n^{-1}/2}.\label{eq:compact2}
\end{eqnarray}

Thus with $h_n = O(\exp\{n^{1/5}\}), l_n = O(n^{-1/5}), r_n =O(n^{1/5}), M_n = O(n^{1/10}\log n)$,
(\ref{eq:compact1}) and (\ref{eq:compact2}) implies
\begin{eqnarray*}
\Pi(\mathcal{F}_n^c)= \exp\{-K_6n^{1/5}\}
\end{eqnarray*}
for some constant $K_6 > 0$ guaranteeing that (\ref{eq2}) is satisfied with $\tilde{\epsilon_n} = n^{-1/5}$.

Also with $\bar{\epsilon}_n = n^{-2/5}(\log n)$, (\ref{eq:condentropy}) is satisfied and it follows from
(\ref{eq:entropy1}) and (\ref{eq:entropy2}) that
\begin{eqnarray*}
\log N(\bar{\epsilon}_n, \mathcal{F}_n, \norm{\cdot}_1) \leq K_7 n^{1/5}(\log n)^2
\end{eqnarray*}
for some constant $K_7 > 0$.

Hence $\max\{\bar{\epsilon}_n, \tilde{\epsilon}_n\} = n^{-2/5}\log n$.

\section{The non-compact case}\label{sec:noncompact}
The analysis of the non-compact case poses greater technical difficulties compared to the compact case, especially in verifying condition (\ref{eq3}).
Recall that in the compact case, $K(f_0, f_{\mu, \sigma}) \precsim O(\sigma^4) + \norm{\mu - \mu_0}_{\infty}^2/\sigma^2$. However, if $\mbox{supp}(f_0) = \mathbb{R}$, then the corresponding quantile function $\mu_0$ has $\norm{\mu_0}_{\infty} = \infty$. This prohibits us from bounding $\int f_0 \log(f_0/f_{\mu, \sigma})^i, i= 1, 2$ using Lemma \ref{lem:logsup}, since no prior for $\mu$ supported on $C[0, 1]$ can concentrate around arbitrarily small neighborhoods of the true quantile function $\mu_0$ in sup-norm. Since the tail behavior of $f_0$ has a one-to-one correspondence with the behavior of $\mu_0$ near the boundary, we make additional assumptions on the tails of $f_0$ similar to (C3) in \citet{kruijer2010adaptive}.
\begin{ass}\label{ass:3}
$f_0$ has exponential tails, i.e., there exist positive constants $T, M, \tau_1, \tau_2$ such that
\begin{align}\label{eq:ass3}
f_0(x) \leq M \exp(-\tau_1 |x|^{\tau_2}), \quad |x| \geq T.
\end{align}
\end{ass}
\begin{rem}
 Remark \ref{rem:ass:1} suggests that under Assumption \ref{ass:3}, $w_0$ behaves like a polynomial near the tails and hence Assumption \ref{ass:1} is automatically satisfied as long as $w_0$ or equivalently $f_0$ is twice continuously differentiable.
\end{rem}

To derive concentration inequalities for the prior on $\mu$, it is convenient to work with a series prior for $\mu$ as follows:
\begin{ass}\label{ass:1pnc}
For an orthonormal basis $\{\phi_{j}\}_{j=0}^{\infty}$ of $L_{2}[0, 1]$, a sequence of scales $\lambda_j \downarrow 0$, a fixed domain-rescaling integer $a$, a global scaling factor $b > 0$ and a truncation level $J$, consider a prior distribution for $\mu$ given by an orthonormal series expansion
\begin{eqnarray*}
W^J(t) =\sum_{j=0}^{J}\lambda_{j}Z_jb\phi_j(at), \, t \in [0, 1].
\end{eqnarray*}
In the sequel we will chose sequences for $J, a$ and $b$ given by $J_n=O(n^{1/5}), b_n = n^{-1/10}(\log n)^{1/2}$ and $a_n = n^{\alpha}$ for some $\alpha > 1$ to attain the optimal rate of convergence.
\end{ass}

\begin{rem}
Let $W_{2, q}[0,1]$ denote the Sobolev space of $L_{2}[0, 1]$ functions $f$ whose weak partial derivative of order $q$, $D^q f \in L_2[0, 1]$. Also, for $C > 0$, denote
\begin{eqnarray}
W_{2, q}^C[0,1] = \{f \in L_{2}[0, 1]: \norm{f}_{2, q}^2 =  \norm{D^{q}f}_{2}^{2} \leq C \}
\end{eqnarray}
to be the set of functions in $W_{2, q}[0,1]$ norm-bounded by $C$.
In the sequel, we shall assume that $\phi_{j}$'s are given by a cosine basis
\begin{eqnarray}
\phi_{0}(t) &=& \frac{1}{\sqrt{2}} \\
\phi_{j}(t) &=& \cos( 2\pi j t), \, j \geq1 \\
\end{eqnarray}
so that for $f(t)= \sum_{j=0}^{\infty}\theta_jb \phi_j(at)$, one has $\norm{f}_{2, q}^2 = b^2\sum_{j=1}^{\infty}\theta_j^2(2\pi aj)^q$.  The techniques used subsequently can be easily extended to other orthonormal bases.
\end{rem}

We are now in a position to state the main theorem of posterior convergence rates for the non-compact case.
 \begin{thm}\label{thm:non-compact}
If $f_0$ is twice continuously differentiable, satisfies Assumptions \ref{ass:2} and \ref{ass:3}, and the priors $\Pi_\mu$ and $\Pi_{\sigma}$ are as in
Assumptions \ref{ass:1pnc} and \ref{ass:2p} respectively, the best obtainable posterior rate of convergence relative to $h$ is
\begin{eqnarray}\label{eq:optrate}
\epsilon_{n} = n^{-\frac{2}{5}}(\log n)^{t_0},
\end{eqnarray}
for some global constant $t_0$.
\end{thm}

The construction of the sieves $\mathcal{F}_n$ is similar to the compact case and we shall omit the details of calculating the entropy and complement probability of $\mathcal{F}_n$ as they are essentially similar to the proof of Theorem \ref{thm:compact}. Verifying the KL condition in \ref{eq3} is the biggest hurdle in the non-compact case; we briefly outline the steps needed to bound the integrals within parenthesis in \ref{eq3}. The basic idea is to separate the integrals into an integral over a compact set and its complement. Inside the compact set, one can replace $f_0$ by a compact approximation $f_{0\sigma}$ and approximate the quantile function $\mu_{0\sigma}$ of $f_{0\sigma}$ by an infinitely smooth function on $[0,1]$ in an appropriate sense, which enables one to obtain the right concentration rate using a smooth prior on $C[0,1]$. The complement term can be handled by exploiting the exponential tails of $f_0$.

To elaborate, first define sets $E_{\sigma} = \{x ~:~ f_0(x) > \sigma^{H_1}\}$, $E_{\sigma}' = \{x ~:~ f_0(x) > \sigma^{H_2}\}$. Clearly, $E_{\sigma} \subset E_{\sigma}'$ if $H_2 > H_1$. Without loss of generality, one can assume $E_{\sigma}' = [d_{\sigma}, e_{\sigma}]$ by Assumption \ref{ass:2}.
Let $g_{0\sigma} = f_0 1_{E_{\sigma}'}$ denote the restriction of $f_0$ to the compact set $E_{\sigma}'$ and let $f_{0\sigma}$ be $g_{0\sigma}$ normalized to make it a density supported on $E_{\sigma}'$. Further, let $\mu_{0\sigma}:[0,1] \to E_{\sigma}'$ denote the quantile function of $f_{0\sigma}$ and denote $f_{\mu_{0\sigma},\sigma} = \phi_{\sigma}*f_{0\sigma}$.

We now bound $\mbox{V}(f_0, f_{\mu, \sigma}) = \int f_0 \log(f_0/f_{\mu, \sigma})^2$, the treatment of $\mbox{KL}(f_0, f_{\mu, \sigma})$ follows similarly. To start with, observe that
\begin{align}\label{eq:decomp1}
\int f_0 \log\bigg(\frac{f_0}{f_{\mu, \sigma}} \bigg)^2 \precsim \int f_0 \log\bigg(\frac{f_0}{f_{\mu_0, \sigma}} \bigg)^2 + \int f_0 \log\bigg(\frac{f_{\mu_0, \sigma}}{f_{\mu, \sigma}} \bigg)^2.
\end{align}
Using $f_{\mu_0, \sigma} = \phi_{\sigma}*f_0$, it follows from Lemma 8 of \citet{ghosal2007posterior} that,
\begin{align*}
\int f_0 \log\bigg(\frac{f_0}{f_{\mu_0, \sigma}} \bigg)^2 \leq h^2(f_0, \phi_{\sigma}*f_0)\bigg(1 + \log \norm{\frac{f_0}{\phi_{\sigma}*f_0}}_{\infty}\bigg)^2.
\end{align*}
Since $h^2(f_0, \phi_{\sigma}*f_0) \precsim O(\sigma^4)$ from \ref{eq:H} and $\phi_{\sigma}*f_0 \geq C f_0$ by Assumption \ref{ass:2}, one has
$\int f_0 \log(f_0/f_{\mu_0,\sigma})^2 \precsim O(\sigma^4)$. To handle the second term in \ref{eq:decomp1}, we break up the integral into integrals over $E_{\sigma}$ and $E_{\sigma}^c$ and further decompose the first term to obtain,
\begin{align}\label{eq:decomp2}
& \int f_0 \log\bigg(\frac{f_{\mu_0, \sigma}}{f_{\mu, \sigma}} \bigg)^2 \precsim \notag \\
&\bigg\{
\int_{E_{\sigma}}f_0 \log\bigg(\frac{f_{\mu_0, \sigma}}{f_{\mu_{0\sigma}, \sigma}} \bigg)^2 +
\int_{E_{\sigma}}f_0 \log\bigg(\frac{f_{\mu_{0\sigma}, \sigma}}{f_{m_{\sigma}, \sigma}} \bigg)^2 +
\int_{E_{\sigma}}f_0 \log\bigg(\frac{f_{m_{\sigma}, \sigma}}{f_{\mu, \sigma}} \bigg)^2
\bigg\} + \int_{E_{\sigma}^c} f_0 \log\bigg(\frac{f_{\mu_0, \sigma}}{f_{\mu, \sigma}} \bigg)^2.
\end{align}
As mentioned before, we work with a compactly supported approximation $f_{0\sigma}$ of $f_0$ on $E_{\sigma}$, with the support $E_{\sigma}'$ of $f_{0\sigma}$ containing $E_{\sigma}$ and exploit the exponentially decaying tails of $f_0$ on $E_{\sigma}^c$. $m_{\sigma}$ in  \ref{eq:decomp2} is an infinitely smooth function whose choice will be made explicit later. We now provide a detailed analysis of each term in \ref{eq:decomp2}.

We start with the last term on the right hand side of \ref{eq:decomp2}. The main idea is to work on a norm-bounded subset of the function space where the density function $f_{\mu, \sigma}$ can be bounded below and utilize the sub-Gaussian tails of $\norm{\mu}_{\infty}$ to bound the integral outside the above region.
Observe that for $\norm{\mu}_{\infty} \leq M$,
\begin{eqnarray*}
f_{\mu, \sigma}(y) \geq \frac{1}{\sqrt{2\pi\sigma^2}}e^{-\frac{1}{2\sigma^2}(\abs{y} + M)^2},
\end{eqnarray*}
which implies
\begin{eqnarray*}
& &\int_{E_{\sigma}^c} f_0(y) \log\bigg(\frac{f_{\mu_0, \sigma}(y)}{f_{\mu, \sigma}(y)} \bigg)^2dy \leq \int_{E_{\sigma}^c}f_0(y) \log \left\{\frac{C_4\sigma}{e^{-1/(2\sigma^2)(y - M)^2}} \right\}dy \\
&\leq& \log C_4\sigma\int_{E_{\sigma}^c}f_0(y)dy + \frac{1}{2\sigma^2}\int_{E_{\sigma}^c}(y- M)^2dy\\
&=&  \bigg(\log C_4 \sigma + \frac{M^2}{2\sigma^2}\bigg)\int_{E_{\sigma}^c} f_0(y)dy + \frac{1}{2\sigma^2}\int_{E_{\sigma}^c}y^2f_0(y) dy -\frac{M}{\sigma^2}\int_{E_{\sigma}^c} y f_0(y) dy.
\end{eqnarray*}

Now since $\int_{E_{\sigma}^c}y^jf_0(y)dy \leq \sigma^{H_1/2}\int_{E_{\sigma}^c}y^j\sqrt{f_0(y)}dy, j= 0,1,2$, we need to choose $H_1$ satisfying
$\sigma^{H_1/2} \approx \sigma^{6}/M$
to make $\int_{E_{\sigma}^c} f_0(y) \log\bigg(\frac{f_{\mu_0, \sigma}(y)}{f_{\mu, \sigma}(y)} \bigg)^2dy  \precsim O(\sigma^4)$.

To bound the integral over the set $\{ \norm{\mu}_{\infty} > M\}$, we provide an upper bound to $P(\norm{W^J}_{\infty} > M) $ in the following Lemma \ref{lem:comp}, with proof in the Appendix.
\begin{lem}\label{lem:comp}
With  $\sigma_{J}^2 = \sum_{j=1}^{J}\lambda_j^2$,
\begin{eqnarray*}
P(\norm{W^J}_{\infty} > M) \leq 2aM \exp\bigg[ -\frac{1}{2b^2\sigma_{J}^2}\bigg\{  M - C_6\frac{1}{M}\{(\log a)^{1/2} + (\log M)^{1/2}\} \bigg\}^2 \bigg]
\end{eqnarray*}
for some constant $C_6 > 0$.
\end{lem}

We now consider the three terms inside the parenthesis in \ref{eq:decomp2}. Let us start with $\int_{E_{\sigma}} f_0 \log(f_{\mu_0, \sigma}/f_{\mu_{0\sigma}, \sigma})^2$. For $y \in E_{\sigma}$, $f_{\mu_0, \sigma}(y)/f_{\mu_{0\sigma}, \sigma}(y) = \phi_{\sigma}*f_0(y)/\phi_{\sigma}*f_{0\sigma}(y)$. Recall, $f_{0\sigma}(y) = f_0 1_{E_{\sigma}'}(y)/\psi_{\sigma}$ with $\psi_{\sigma} = \int_{E_{\sigma}'} f_0(y)$. Note that
\begin{eqnarray*}
\psi_{\sigma} \geq 1- \int_{E_{\sigma}'}f_0(x)dx &\geq& 1- \sigma^{H_2/2}\int_{E_{\sigma}'}\sqrt{f_0(x)}dx \\
&\geq& 1- \sigma^{4}
\end{eqnarray*}
for $H_2 \geq 8$.
Now,
\begin{align*}
& \phi_{\sigma}*f_0(y) = \int \phi_{\sigma}(y-t) f_0(t) dt \\
& = \int_{E_{\sigma}'} \phi_{\sigma}(y-t) f_0(t) dt + \int_{(E_{\sigma}')^c} \phi_{\sigma}(y-t) f_0(t) dt.
\end{align*}
Hence,
\begin{align*}
\frac{\phi_{\sigma}*f_0(y)}{\phi_{\sigma}*f_{0\sigma}(y)} = \psi_{\sigma} \bigg\{ 1 +
\frac{\int_{(E_{\sigma}')^c} \phi_{\sigma}(y-t) f_0(t) dt}{\int_{E_{\sigma}'} \phi_{\sigma}(y-t) f_0(t) dt} \bigg\}.
\end{align*}
Now, for $t \in (E_{\sigma}')^c$ and $y \in E_{\sigma}$, $f_0(t) \leq \sigma^{H_2} \leq \sigma^{H_2 - H_1} f_0(y)$, implying
$\int_{(E_{\sigma}')^c} \phi_{\sigma}(y-t) f_0(t) dt \leq \sigma^{H_2 - H_1} f_0(y)$. Moreover,
\begin{align*}
& \int_{E_{\sigma}'} \phi_{\sigma}(y-t) f_0(t) dt = \phi_{\sigma}*f_0(y) - \int_{(E_{\sigma}')^c} \phi_{\sigma}(y-t) f_0(t) dt  \\
& \geq C f_0(y) - \sigma^{H_2 - H_1} f_0(y).
\end{align*}
Thus,
\begin{align}\label{eq:term1a}
\frac{\phi_{\sigma}*f_0(y)}{\phi_{\sigma}*f_{0\sigma}(y)} \precsim \psi_{\sigma}\bigg(1 + \frac{\sigma^{H_2 - H_1}}{C - \sigma^{H_2 - H_1}}\bigg)
\precsim 1
\end{align}
On the other hand,
\begin{align}\label{eq:term1b}
& \frac{\phi_{\sigma}*f_0(y)}{\phi_{\sigma}*f_{0\sigma}(y)} = \psi_{\sigma} \frac{\phi_{\sigma}*f_0(y)}{\int_{E_{\sigma}'} \phi_{\sigma}(y-t) f_0(t) dt} \notag \\
& \geq \psi_{\sigma} = 1 + O(\sigma^4).
\end{align}
Hence, from \ref{eq:term1a} \& \ref{eq:term1b}, one has $\phi_{\sigma}*f_0(y)/\phi_{\sigma}*f_{0\sigma}(y) = 1 + O(\sigma^4)$ for $y \in E_{\sigma}$, implying
$\int_{E_{\sigma}} f_0 \log(\phi_{\sigma}*f_0/\phi_{\sigma}*f_{0\sigma})^2 = O(\sigma^4)$.

We next turn our attention to $\int_{E_{\sigma}}f_0 \log(f_{\mu_{0\sigma}, \sigma}/f_{m_{\sigma}, \sigma})^2$. \citet{ghosal2007posterior} showed that a Gaussian convolution of a compactly supported distribution can be approximated with high accuracy by a finite mixture of normals with ``relatively few'' mixture components and \citet{kruijer2010adaptive} obtained a finer calibration of their result to handle the above integral. However, it becomes unwieldy to use their result in our setup since their approximating density is obtained as the convolution of a Gaussian kernel with a discrete mixing distribution with finitely many support points, with the corresponding quantile function being a step function on $[0,1]$ with finitely many jumps. Although one can place a prior on $\mu$ whose realizations are step functions, several issues arise with the posterior computation including choosing \& updating the number of steps. It would be appealing to use a smooth prior on $\mu$ and yet obtain a similar approximation result. We borrow techniques from the physics literature on maximum entropy moment matching or MAXENT \citep{mead1984maximum} to develop an approximation result with a smooth mixing measure as follows:
\begin{lem}\label{lem:maxent}
Let $f$ be a density compactly supported on $[-a_{\sigma}, a_{\sigma}]$ with $a_{\sigma} = a_0 |\log \sigma|^{1/\tau_2}$ with $\sigma$ small enough. Then, for any $A_0 > 0$, there exists an infinitely smooth density $f_{m_{\sigma}}$ on $[-a_{\sigma}, a_{\sigma}]$, such that $\norm{\phi_{\sigma}*f - \phi_{\sigma}*f_{m_{\sigma}}}_{\infty} \leq \sigma^{-1} \exp(-C |\log \sigma |^{1/\tau_2})$.
\end{lem}
A proof of Lemma \ref{lem:maxent} can be found in the Appendix.

The tail behavior of $f_0$ implied by Assumptions \ref{ass:2} \& \ref{ass:3} imply $E_{\sigma}' \subset [-a_{\sigma}, a_{\sigma}]$ with $a_{\sigma} = a_0 |\log \sigma|^{1/\tau_2}$ with $a_0 = \big(\frac{H_2}{2\tau_1}\big)^{1/\tau_2}$.

Let $m_{\sigma}$ be the quantile function of the compactly supported density $f_{m_{\sigma}}$ in Lemma \ref{lem:maxent} and let $f_{m_{\sigma}, \sigma} = \phi_{\sigma}*f_{m_{\sigma}}$. Note that for $y \in E_{\sigma}'$,
\begin{eqnarray*}
f_{\mu_{0\sigma}, \sigma}(y) &=& \frac{1}{\psi_{\sigma}}\int_{E_{\sigma}'}\phi_{\sigma}(y-t) f_{0}(t)dt \nonumber \\ &\geq&  \frac{(C- \sigma^{H_2-H_1})}{1-\sigma^4}f_0(y) \geq \frac{C}{2} \sigma^{H_2},
\end{eqnarray*}
for sufficiently small $\sigma$. Using $\abs{\log x} \leq \max\{\log \abs{x-1}, \log \abs{1/x-1}\}$ for $x > 0$, one gets
\begin{align*}
& \int_{E_{\sigma}} f_0 \log\bigg(\frac{f_{\mu_{0\sigma}, \sigma}}{f_{m_{\sigma}, \sigma}} \bigg)^2
\leq \int_{E_{\sigma}} f_0 \bigg(\frac{\norm{f_{\mu_{0\sigma}, \sigma} - f_{m_{\sigma}, \sigma}}_{\infty}}{(C/3)\sigma^{H_2}} \bigg)^2,
\end{align*}
for $A_0$ large enough.  By choosing $A_0$ sufficiently large and using Lemma \ref{lem:maxent}, we obtain
\begin{eqnarray}\label{eq:main1}
\bigg(\frac{\norm{f_{\mu_{0\sigma}, \sigma} - f_{m_{\sigma}, \sigma}}_{\infty}}{(C/3)\sigma^{H_2}} \bigg)^2 \precsim O(\sigma^4).
\end{eqnarray}

Finally, we consider the third term $\int f_0 (\log f_{m_{\sigma}, \sigma} / f_{\mu, \sigma})^2$ inside the parenthesis in \ref{eq:decomp2}.
Proceeding as in the previous case, we first lower bound $f_{m_{\sigma}, \sigma}$ on $E_{\sigma}'$. In the previous case, we already obtained $f_{\mu_{0\sigma}, \sigma} \succsim \sigma^{H_2}$. From  \cite{borwein1991convergence}, $\norm{f_{m_{\sigma}} - f_{0\sigma}}_{\infty} = o(k^{-1})$ if we match $k$ moments.  From Lemma  \ref{lem:maxent}, we know that $k \approx O(\sigma^{-\alpha}\abs{\log \sigma}^{\alpha/\tau_2})$ and hence by choosing $\alpha$ large enough, we can make $f_{m_{\sigma}, \sigma} \succsim \sigma^{H_2}$ on $E_{\sigma}'$. Now, using the same bound for $\abs {\log x}$ as before, we need $\norm{f_{m_{\sigma,\sigma}} - f_{\mu,\sigma}}_{\infty}$ to be $O(\sigma^H)$ for any $H > 0$. From \cite{kruijer2010adaptive} it follows that $\sup_{y}\abs{\phi_{\sigma}(y-\mu_1) - \phi_{\sigma}(y-\mu_2)} \precsim \frac{\abs{\mu_1 -  \mu_2}}{\sigma^2}$ so that
\begin{eqnarray}\label{eq:main2}
\norm{f_{m_{\sigma,\sigma}} - f_{\mu,\sigma}}_{\infty}\precsim \frac{\norm{\mu - m_{\sigma}}_{\infty}}{\sigma^2}.
\end{eqnarray}

We shall now exploit the infinite differentiability of $m_{\sigma}$ to view it as an element of $W_{2,q}$ for some large $q$ and calculate the probability of the Gaussian process $W^J$ concentrating around $m_{\sigma}$.

The reproducing kernel Hilbert space (RKHS) of $W^J$ consists of the set of functions $w(t) = \sum_{j=0}^{J}w_jb\phi_j(at), t \in [0, 1]$ with RKHS norm
\begin{eqnarray*}
\norm{w}_{\mathbb{H}}^2 = \sum_{j=0}^{J}\frac{w_j^2}{\lambda_j^2} .
\end{eqnarray*}
Since $m_{\sigma}$ is infinitely differentiable, $m_{\sigma} \in W_{2, q}^{C}[0,1]$ for any $q \geq 1$ (with $C$ possibly depending on $q$). Hence there exists $\{m_{\sigma, j}\}_{j=0}^{\infty}$ such that
\begin{eqnarray*}
m_{\sigma}(t) = \sum_{j=0}^{\infty}m_{\sigma,j}b\phi_j(at), \, t \in [0,1].
\end{eqnarray*}
Consider the projection of $m_{\sigma}$ on the RKHS of $W^J$ as
\begin{eqnarray*}
m_{\sigma}^{J}(t) =\sum_{j=0}^{J}m_{\sigma,j}b\phi_j(at).
\end{eqnarray*}
In the sequel, we will choose a $q \geq 1$ and sequences of integers $J_n \uparrow \infty$, $a_n, b_n$ to achieve the optimal rate of convergence.
To that end, we calculate the prior concentration of $W^J$ around $m_{\sigma}^{J}$ for a fixed $J$ with $\lambda_j = j^{-q/4}$ for $j \geq 1$. Recall that the prior concentration function of the Gaussian process $W^{J}$ around  $m_{\sigma}^{J}$ is given by
\begin{eqnarray}
\phi_{m_{\sigma}^{J}}(\epsilon) = \inf_{h \in \mathbb{H}: \norm{h - m_{\sigma}^J}_{\infty} < \epsilon}\norm{h}^2_{\mathbb{H}} - \log \mbox{P}(\norm{W^J}_{\infty} \leq \epsilon).
\end{eqnarray}
\begin{lem} \label{lem:priorconc}
For $q > 16$,
\begin{eqnarray*}
\phi_{m_{\sigma}^{J}}(\epsilon) \precsim  \frac{1}{a^qb^2}\norm{m_{\sigma}^{J}}_{2,q}^2 +
\begin{cases}
J \big(1 + \log \frac{b}{\epsilon}\big), \, \epsilon J^{q/4} \precsim J^2 \\
\big(\frac{b}{\epsilon}\big)^{20/q},  \,  \epsilon J^{q/4} \geq J^2.
\end{cases}
\end{eqnarray*}
\end{lem}
For a proof, refer to the Appendix.

Recall that we need the concentration bound for $m_{\sigma}$, while in the above lemma, we obtained the concentration bound for $m_{\sigma}^J$. We thus need error bounds on how well the truncation $m_{\sigma}^J$ approximates the function $m_{\sigma}$. Noting that $m_{j,\sigma}^2 \leq \frac{1}{b^2}\norm{m_{\sigma}}^{2}_{2,q}(aj)^{-q}$,
\begin{eqnarray}\label{eq:final}
\norm{m_{\sigma} - m_{\sigma}^J}_{\infty} &\leq& \norm{\sum_{j=J+1}^{\infty}m_{j,\sigma}\phi_j(t)}_{\infty} \nonumber \\
&\leq& b\sum_{j=J+1}^{\infty}\abs{m_{j, \sigma}} \nonumber \\
&\leq& \norm{m_{\sigma}}_{2,q}(aJ)^{-(q/2-1)}.
\end{eqnarray}

To bound the final term in (\ref{eq:final}), we provide an upper bound to $\norm{m_{\sigma}}_{2,q}^2$ in the following Lemma \ref{lem:upperbd}.

\begin{lem}\label{lem:upperbd}
$\norm{m_{\sigma}}_{2,q}^2 \leq \sigma^{-(2q-1)H_2}. $
\end{lem}
\begin{proof}
Exact derivation of the bound is quite tedious and we shall only sketch the main steps of the proof.
Recall that  $m_{\sigma}(x) = F_{m_{\sigma}}^{-1}(x), x \in [0, 1]$ where $F_{m_{\sigma}} : [-a_{\sigma}, a_{\sigma}] \to [0, 1]$ given by $F_{m_{\sigma}}(x) = \int_{-a_{\sigma}}^{x}f_{m_{\sigma}}(t)dt$.
Then
\begin{eqnarray}
\norm{m_{\sigma}}_{2,q}^2 =  \int_{0}^{1}\{(F_{m_{\sigma}}^{-1})^{(q)}(x)\}^2dx
\end{eqnarray}
Observe that
\begin{eqnarray*}
(F_{m_{\sigma}}^{-1})'(1) & = &\frac{1}{f_{m_{\sigma}}(a_{\sigma})},  (F_{m_{\sigma}}^{-1})''(1)  =  -\frac{f_{m_{\sigma}}'(a_{\sigma})}{f_{m_{\sigma}}(a_{\sigma})^3}  \\
(F_{m_{\sigma}}^{-1})'''(1) & = &-3\frac{f_{m_{\sigma}}'(a_{\sigma})^2}{f_{m_{\sigma}}(a_{\sigma})^5} - \frac{f_{m_{\sigma}}''(a_{\sigma})}{f_{m_{\sigma}}(a_{\sigma})^4}.
\end{eqnarray*}
Proceeding like this, one has $\{f_{m_{\sigma}}(a_{\sigma})\}^{2q-1}$ in the denominator for $(F_{m_{\sigma}}^{-1})^{(q)}(1)$ and $\{f_{m_{\sigma}}(-a_{\sigma})\}^{2q-1}$  for $(F_{m_{\sigma}}^{-1})^{(q)}(0)$. The numerator terms of the above expression are bounded.  From \cite{borwein1991convergence}, we know that $\norm{f_{m_{\sigma}} - f_{0\sigma}}_{\infty} = o(k^{-1})$ if we match $k$ moments.  From Lemma  \ref{lem:maxent}, $k \approx O(\sigma^{-\alpha}\abs{\log \sigma}^{\alpha/\tau_2})$ and hence by choosing $\alpha$ large enough we can make $f_{m_{\sigma}}(a_{\sigma}) \geq f_{0\sigma}(a_{\sigma}) - \sigma^{H_2+1}$ and  $f_{m_{\sigma}}(-a_{\sigma}) \geq f_{0\sigma}(-a_{\sigma}) - \sigma^{H_2+1}$ which implies,
\begin{eqnarray}
(F_{m_{\sigma}}^{-1})^{(q)}(1) \precsim \sigma^{-(2q-1)H_2}, \, (F_{m_{\sigma}}^{-1})^{(q)}(0) \precsim \sigma^{-(2q-1)H_2}.
\end{eqnarray}
Noting that $\int_{0}^{1}\{(F_{m_{\sigma}}^{-1})^{(q)}(x)\}^2dx \precsim \max \{\{(F_{m_{\sigma}}^{-1})^{(q)}(0)\}^2, \{(F_{m_{\sigma}}^{-1})^{(q)}(1)\}^2 \}$,
the conclusion follows immediately.
\end{proof}

We are now in a position to complete the proof of Theorem \ref{thm:non-compact}.
\subsection{Proof of Theorem \ref{thm:non-compact}}
\begin{proof}
Since we can bound the numerator of the rhs of \ref{eq:main2} as
\begin{eqnarray}
\norm{W^J - m_{\sigma}}_{\infty} \leq  \norm{W^J - m_{\sigma}^J}_{\infty} + \norm{m_{\sigma} - m_{\sigma}^J}_{\infty},
\end{eqnarray}
we need $\norm{m_{\sigma} - m_{\sigma}^J}_{\infty}=\norm{m_{\sigma}}_{2,q}(aJ)^{-(q/2-1)}$
 to be $O(\sigma^{H_2+4})$ so that the fourth term of \ref{eq:decomp2} is $O(\sigma^4)$.
Next, we calculate the prior probability $P(\norm{W^J - m_{\sigma}^J}_{\infty} \leq \sigma^{H_2+4})$. Using Lemma \ref{lem:priorconc}, we can see that if
\begin{align}
M_n\sigma_n^{H_1/2}  &= O(\sigma_n^6), \label{eq:thmcond-2}\\
P(\norm{W^J}_{\infty} > M_n ) &= O(e^{-1/\sigma_n}), \label{eq:thmcond-1}\\
\norm{m_{\sigma}}_{2,q}(a_nJ_n)^{-(q/2-1)} &= O(\sigma_n^{H_2+4}), \label{eq:thmcond0}\\
(b_n\sigma_n^{-(H_2+4)})^{20/q} &= O(\sigma_n^{-1}), \label{eq:thmcond1} \\
\frac{1}{a_n^qb_n^2}\norm{m_{\sigma}^{J}}_{2,q}^2 &= O(\sigma_n^{-1}), \label{eq:thmcond2} \\
J_n &=O(\sigma_n^{-1}), \label{eq:thmcond3}
\end{align}
and
$\sigma \in [\sigma_n, 2\sigma_n]$, then
\begin{align*}
&\Pi\bigg( f_{\mu, \sigma}:  \int f_0 \log \frac{f_0}{f_{\mu, \sigma}} \leq \sigma_n^4, \int f_0 \log \bigg(\frac{f_0}{f_{\mu, \sigma}}\bigg)^2 \leq \sigma_n^4 \bigg) \\ &\geq  P(\sigma \in [\sigma_n, 2\sigma_n], \norm{W^J - m_{\sigma}^J}_{\infty} \leq \sigma_n^{H_2+4}, \norm{W^J}_{\infty}\leq M_n) \\ \geq
O(\sigma_n^{-1}\abs{\log \sigma_n}).
\end{align*}

Next we make specific choices for $q, a_n, b_n, M_n$ and  $\sigma_n$.   Clearly $\sigma_n = O(n^{-1/5})$ for the 
optimal rate.  \ref{eq:thmcond-2}-\ref{eq:thmcond3} determine the values of the sequences  $M_n, a_n, b_n$ and $q$ using the upper bounds on $\norm{m_{\sigma}}_{2,q}$ and $P(\norm{W^J}_{\infty} > M_n )$  provided in Lemma \ref{lem:upperbd} and Lemma \ref{lem:comp} respectively. It can be verified that \ref{eq:thmcond0}, \ref{eq:thmcond1} and \ref{eq:thmcond2} are satisfied with $q$ greater than the positive root of the quadratic equation
\begin{eqnarray}
(9/10)q^2 -(10H_2 +24/5)q - 2(2H_2+7) =0,
\end{eqnarray}
which is satisfied by $q \approx 95$.  Choosing  $q=150$,  $H_1 =  H_2 =12$,  $a_n=n^{\alpha}$ for some $\alpha > 1$,
$M_n^2 = O(\log n), b_n=O(n^{-1/10}(\log n)^{1/2})$, we can see that the convergence rate is $\epsilon_n= n^{-2/5}\log^{t_0} n$ for some global constant $t_0$.
\end{proof}

\section{Special cases}\label{sec:splcases}
A desirable property of any nonparametric model is that it can ``collapse'' back to a simpler structure when the additional flexibility is not warranted.  For example, the nonparametric prior may be centered on a smaller class of densities (e.g., a parametric family), with a faster rate of convergence obtained when the true density falls within this smaller class.  In this section, we study such collapsing behavior in a couple of cases.

\subsection{Properly centering the prior leads to parametric rate of convergence}
We have already noted in the introduction that one can center the non-parametric model on a parametric family by centering the prior on the transfer function on a parametric class of quantile (or inverse c.d.f.) functions $\{F_{\theta}^{-1} ~:~ \theta \in \Theta\}$. Here we show that if our guess about the true density $f_0$ is correct,  we can actually achieve a parametric rate of convergence by centering the prior for $\mu$ on $F_{0}^{-1}$.  Centering the prior on the true quantile function expands the RKHS to include the best approximation which is the true quantile function itself.   
We formalize our result in the following Theorem \ref{thm:spcl1}.

\begin{ass}\label{ass:1pr}
Define $\mu_0$ to be $F_{0}^{-1}$. We assume $\mu$ follows a Gaussian process $\mbox{GP}(\mu_0, c)$ centered at $\mu_0$ and with a squared exponential covariance kernel $c(\cdot, \cdot; A)$ and a Gamma prior for the inverse-bandwidth $A$, so that,
$c(t, s; A) = e^{-A(t-s)^2}, t,s \in [0, 1], A \sim \mbox{Ga}(p,q)$.
\end{ass}

\begin{ass}\label{ass:2pr}
$\frac{\sigma^2}{n^{-1/4}\log n} \sim \mbox{IG}(a,b)$.
\end{ass}

 \begin{thm}\label{thm:spcl1}
If $f_0$ is compact and satisfies Assumption \ref{ass:1} and the priors $\Pi_\mu$ and $\Pi_{\sigma}$ are as in
Assumptions \ref{ass:1pr} and Assumption \ref{ass:2pr} respectively with the correct centering $f_0$, the best obtainable posterior rate of convergence relative to $h$ is
\begin{eqnarray}\label{eq:optrate}
\epsilon_{n} = n^{-\frac{1}{2}}(\log n)^{t_0}.
\end{eqnarray}
for some global constant $t_0$.
\end{thm}
\begin{proof}
The portion from which the proof differs from the proof of Theorem \ref{thm:compact} is the calculation of the prior concentration.
Let $\mu = \mu_0 + W$ where $W \sim \mbox{GP}(0, c)$.
It is easy to see that
\begin{eqnarray*}
P(\norm{\mu - \mu_0}_{\infty} \leq 2\epsilon)  = P(\norm{W}_{\infty} \leq 2\epsilon) \geq e^{-C_1a_1\log^2(a_1/\epsilon)}P(a_0 < A < a_1).
\end{eqnarray*}
Hence with $\sigma_n = n^{-1/4}, l_n =n^{-1/4}, h_n = n^{\beta}$ for some $\beta > 0$, we can show that
\begin{eqnarray*}
P(\sigma \in [\sigma_n, 2\sigma_n]) \geq \exp\{ -(\log n)^{t_1}\}
\end{eqnarray*}
for some $t_1 > 0$.

Define the Gaussian process sieve to be 
\begin{eqnarray*}
B_n = f_0 + \bigg[\bigg(M_n\sqrt{\frac{r_n}{\xi_n}}\mathbb{H}_1^r + \bar{\delta}_n\mathbb{B}_1\bigg) \cup
\bigg(\cup_{a < \xi_n}(M_n \mathbb{H}^a_1) + \bar{\delta}_n\mathbb{B}_1 \bigg)\bigg],
\end{eqnarray*}
where the sequences $\xi_n, \bar{\delta}_n, M_n$ are exactly as specified in the proof of Theorem \ref{thm:compact}. 
It follows from the proof of Theorem \ref{thm:compact} that  $\epsilon_n = n^{-1/2} (\log n)^{t_0}$ for some constant $t_0 > 0$.
\end{proof}
\begin{rem}
The extension to the case when the true density is actually non-compact and satisfies the tail condition in Assumption \ref{ass:3} can be handled following the steps in the proof of Theorem \ref{thm:non-compact} with $\Pi_\mu$ in Assumption \ref{ass:1pnc} centered at the non-compact $f_0$ for appropriate choices of sequences $a_n, b_n$ and $\lambda_j$ and $\Pi_{\sigma}$ as in Assumption \ref{ass:2p}.  However one can show that the best obtainable rate of convergence can only be made arbitrarily close to the parametric rate in the sense that the rate of convergence would be slower compared to the parametric rate by a factor $n^{\beta}$ where $\beta$ can be arbitrarily small. 
\end{rem}

\begin{rem}
A more interesting and practical extension of Theorem \ref{thm:spcl1} is the case where one has correctly guessed a parametric family which contains the true density. Suppose the parametric family is given by $\{f_{\theta}: \theta \in \mathbb{R}^p\}$ indexed by the parameter $\theta$ living in some Euclidean space, and the true density $f_0 = f_{\theta_0}$. It is natural then to center the prior for $\mu$ on $F_{\theta}^{-1}$, with a hyperprior on $\theta$ quantifying uncertainty about the value of the finite-dimensional parameter $\theta$. A straightforward application of Theorem \ref{thm:spcl1} shows that it is possible to attain parametric rate of convergence under the same assumptions as in Theorem \ref{thm:spcl1} if the prior on $\theta$ has full support on $\mathbb{R}^p$ and the mapping $\theta \to f_{\theta}$ satisfies mild regularity conditions, e.g., as in \citet{ghosal2000convergence}. In particular, one obtains the near parametric rate if $\theta \sim N_{p}(\mu_0, \Sigma_0)$. If the prior guess about the parametric family is incorrect, then one would still get the near minimax rate for the class of twice continuously differentiable densities. 
\end{rem}

\subsection{True density is a Gaussian convolution with a finite mixture of truncated Gaussians}
\cite{ghosal2001entropies} showed that when the true density is a location-scale or location mixture of normals $\int \frac{1}{\sigma}\phi\left( \frac{y-\mu}{\sigma}\right)F_0(\mu, \sigma)$ with the scale parameter lying between two fixed numbers and the mixing distribution $F_0$ being either compactly supported  or having sub-Gaussian tails, a Dirichlet process mixture of normals can achieve near-parametric rate of convergence.  To mimic the above super-smooth case for our non-linear latent variable model, we shall consider a simplistic situation when the true density is a Gaussian convolution with a finite mixture of truncated Gaussians with the same truncation bounds.  We show below that the rate of convergence in that case can be as close as possible to the parametric rate.  The actual super-smooth case would be the situation when the true density is a finite mixture of Gaussians, so that it can be expressed as a convolution with a finite mixture of Gaussians.  Remark \ref{rem:supersmooth} discusses very briefly about that case.
\begin{thm}
Given any $\alpha > 0$.
If $f_0$ is $\phi_{\sigma_0} * f_1$ where $f_1$ is a finite mixture of truncated Gaussians with the same truncation bounds and  the prior  $\Pi_\mu$ is as in Assumptions \ref{ass:1p} and $\frac{\sigma}{n^{-1/(2\alpha +1)}\log n} \sim \mbox{IG}(a,b)$ respectively, then the best obtainable rate of posterior convergence is
\begin{eqnarray}
n^{-\frac{\alpha}{2\alpha + 1}} \log^{t_{\alpha}}(n)
\end{eqnarray}
where $t_{\alpha}$ is a constant depending on $\alpha$.
\end{thm}
\begin{proof}
Clearly $f_1$ is an infinitely smooth density which has quantile function $\mu_1=F_1^{-1}$ infinitely smooth in $[0, 1]$ and hence $ \mu_1 \in C^{\alpha}([0, 1])$ for any $\alpha > 0$.
Observe that
\begin{eqnarray*}
h^{2}(f_0, f_{\mu, \sigma}) &=& h^{2}(\phi_{\sigma_{0}}* f_1, f_{\mu, \sigma}) \\
 &=& h^{2}(f_{\mu_1, \sigma_{0}}, f_{\mu, \sigma_{0}})  + h^{2}(f_{\mu, \sigma_{0}}, f_{\mu, \sigma}) \\
 &\precsim& \norm{\mu - \mu_1}_{\infty}^2 + \frac{\abs{\sigma_0 - \sigma}}{\sigma}
\end{eqnarray*}
From \cite{van2009adaptive}, we obtain $P(\norm{\mu- \mu_1}_{\infty} < \sigma_n^{\alpha}) \geq  \exp(-\sigma_n^{-1})$ and
$P( \abs{\sigma_0 - \sigma} < \sigma_n^{2\alpha +1}) \geq \exp(-\sigma_n^{-1})$ for $\sigma_n = n^{-1/(2\alpha + 1)}$.
With the same sieve as in the proof of Theorem \ref{thm:compact}, it follows that  $\epsilon_n = n^{-\alpha/(2\alpha +1)} (\log n)^{t_{\alpha}}$ for some constant $t_{\alpha} > 0$.
\end{proof}

\begin{rem} \label{rem:supersmooth}
The extension to the case when the true density is a finite mixture of Gaussians can be handled following the steps in the proof of Theorem \ref{thm:non-compact} with $\Pi_\mu$ and $\Pi_{\sigma}$ in Assumptions \ref{ass:1pnc} and Assumption \ref{ass:2p} respectively for appropriate choices of sequences $a_n, b_n$ and $\lambda_j$ respectively.
\end{rem}

\section{Discussion}\label{sec:discussion}
Non-linear latent variable models offer a flexible modeling framework in a broad variety of problems and improved practical performance has been demonstrated by \citet{lawrence2004gaussian,lawrence2005probabilistic,lawrence2007hierarchical,ferris2007wifi,kundu2011bayesian} among others. The univariate density estimation model studied here can be extended to multivariate density estimation, latent factor modeling and density regression problems; we are currently studying theoretical properties of these extensions building upon the results developed in this article in the baseline case.

In standard Gaussian process regression, the regression function is assumed to be continuous on a compact domain, and one can use standard results on concentration bounds for Gaussian processes \citep{van2008reproducing}. However, we cannot use these results directly as the quantile function of a density supported on the entire real line is unbounded near zero and one. To address this problem, we required assumptions on the tails of the true density and exploited the interplay between the tails of a density and the boundary behavior of the corresponding quantile function. Building a sequence of compact approximations to the true density, accurate concentration bounds around the corresponding quantile functions (which are in $C[0, 1]$) are developed for the Gaussian process prior on the transfer function. While deriving this bound, one has to carefully calibrate the rate at which the RKHS norms of the sequence of approximating quantile functions increase to infinity. A truncated series prior is convenient for this purpose, however one needs to appropriately rescale the prior as in \ref{ass:1pnc} for optimal rate.  It would be interesting to study whether one obtains the same for a host of other commonly used Gaussian process priors. It should be noted here that posterior consistency with commonly used Gaussian process priors is immediate using our treatment of the non-compact case.

We finally note that although our results assume twice continuously differentiability of the true density, one can obtain optimal rate of convergence for arbitrary degree of smoothness of the truth. From the treatment of $\int f_0 \log(f_0 / f_{\mu, \sigma})^2$ in (\ref{eq:decomp2}), it is not difficult to see that all the terms barring $\int f_0 \log(f_0 / f_{\mu_0, \sigma} )^2$ can be made $O(\sigma^H)$ for arbitrarily large $H$. The first term $\int f_0 \log(f_0 / f_{\mu_0, \sigma} )^2$ cannot be improved beyond $O(\sigma^4)$ even if the true density is more than twice continuously differentiable, since $h(f_0, \phi_{\sigma}*f_0)$ can only be $O(\sigma^2)$. This is a well-known issue with Gaussian convolutions and one can improve the approximation bound by using a higher order kernel $\psi_{\sigma}$ \citep{fan1992bias,marron1992exact}, so that $\norm{ f_0 - \psi_{\sigma}*f_0 } = O(\sigma^H)$ for $H$ arbitrarily large. A thorny issue with using higher-order kernels in the frequentist literature is that $\psi_{\sigma}*f_0$ is not guaranteed to be positive everywhere, but one can bypass that easily in a Bayesian framework as one only needs to show that the prior support contains densities that are appropriately close to the true density. Letting $\phi_{\sigma}*f_1 = \psi_{\sigma}*f_0$, one can solve for $f_1$ using inverse-Fourier transforms and one has
$\norm{ f_0 - \phi_{\sigma}*f_1 } = O(\sigma^H)$. Although $ \int f_1 = 1$, $f_1$ can be negative at some places. \citet{kruijer2010adaptive} showed that under suitable conditions on the true density, one obtains the same approximation error for $f_2$, the positive part of $f_1$ normalized to integrate to one. \citet{kruijer2010adaptive} used the twicing kernel method \citep{newey2004twicing} to obtain $f_1$ in a closed-form, one can use the same trick here or use other higher-order kernels to obtain $f_1$. One can then simply replace $\mu_0$ with the quantile function of $f_2$ and proceed with the rest of the analysis identically.

\appendix

\section*{Appendix}
\subsection{Proof of Theorem \ref{thm:support}}
\begin{proof}
Let $f_0$ be a density with quantile function $\mu_0$ that satisfies the conditions of Theorem \ref{thm:support}. Observe that $\norm{\mu_0}_1 = \int_{t=0}^1 \abs{\mu_0(t)} dt = \int_{-\infty}^{\infty} \abs{z} f_0(z) dz < \infty$ since $f_0$ has a finite first moment, and thus $\mu_0 \in \mbox{L}_1[0, 1]$. Fix $\epsilon > 0$. We want to show that $\Pi\{ B_{\epsilon}(f_0) \} > 0$, where $B_{\epsilon}(f_0) = \{f ~:~ \norm{f - f_0}_1 < \epsilon\}$.

Note that $\mu_0 \notin C[0, 1]$, so that $\mbox{pr}( \norm{\mu - \mu_0}_{\infty} < \epsilon)$ can be zero for small enough $\epsilon$. The main idea is to find a continuous function $\tilde{\mu_0}$ close to $\mu_0$ in $L_1$ norm and exploit the fact that the prior on $\mu$ places positive mass to arbitrary sup-norm neighborhoods of $\tilde{\mu_0}$. The details are provided below.

Since $\norm{\phi_{\sigma}*f_0 - f_0}_1 \to 0$ as $\sigma \to 0$, find $\sigma_1$ such that $\norm{\phi_{\sigma}*f_0 - f_0}_1 < \epsilon/2$ for $\sigma < \sigma_1$. Pick any $\sigma_0 < \sigma_1$.
Since $C[0, 1]$ is dense in $\mbox{L}_1[0, 1]$, for any $\delta > 0$, we can find a continuous function $\tilde{\mu_0}$ such that $\norm{\mu_0 - \tilde{\mu_0}}_1 < \delta$. Now, $\norm{ f_{\mu, \sigma} - f_{\tilde{\mu_0}, \sigma} }_1 \leq C \norm{ \mu - \tilde{\mu_0} }_1/\sigma$ for a global constant $C$. Thus, for $\delta = \epsilon \sigma_0/4$,
\begin{align*}
\big\{f_{\mu, \sigma} ~:~ \sigma_0 < \sigma < \sigma_1, \norm{\mu - \tilde{\mu_0}}_{\infty} < \delta \big\} \subset \big\{f_{\mu, \sigma} ~:~ \norm{f_0 - f_{\mu, \sigma}}_1 < \epsilon \big\},
\end{align*}
since $\norm{f_0 - f_{\mu, \sigma}}_1 < \norm{f_0 - f_{\mu_0, \sigma}}_1 + \norm{f_{\mu_0, \sigma} - f_{\tilde{\mu_0}, \sigma}}_1 + \norm{f_{\tilde{\mu_0}, \sigma} - f_{\mu, \sigma}}_1$ and $f_{\mu_0, \sigma} = \phi_{\sigma}*f_0$. Thus, $\Pi\{ B_{\epsilon}(f_0) \} > \mbox{pr}(\norm{\mu - \tilde{\mu_0}}_{\infty} < \delta) ~ \mbox{pr}(\sigma_0 < \sigma < \sigma_1) > 0$, since $\Pi_{\mu}$ has full sup-norm support and $\Pi_{\sigma}$ has full support on $[0, \infty)$.
\end{proof}

\subsection{Proof of Lemma \ref{lem:comp}}
\begin{proof}
From Theorem 5.2 of \cite{adler1990introduction} it follows that if $X$ is a centered Gaussian process on a compact set $T \subset \mathbb{R}^d$ and   $\sigma_{T}^2$ is the maximum variance attained by the Gaussian process on $T$, then for large $M$,
\begin{eqnarray}\label{eq:comp1}
P(\norm{X}_{\infty} > M ) \leq 2N(1/M, T, \norm{\cdot}) \exp\bigg[-\frac{1}{2\sigma_T^2}\{M - \nu(M)\}^2\bigg],
\end{eqnarray}
where  $\nu(M) = C_5\int_{0}^{1/M}
\{\log N(1/M, T, \norm{\cdot})\}^{1/2}d(1/M)$ for some constant $C_5 > 0$.
Observe that $W^J$ is rescaled to $T=[0, a]$ and the maximum variance attained by $W^J$ is $b^2\sigma_{J}^2$. Note that $N(1/M, T, \norm{\cdot}) = aM$. Now
\begin{eqnarray*}
\nu(M) &\leq& C_6 \int_{0}^{1/M}\{\log(aM)\}^{1/2}d(1/M) \\
&\leq& C_6\int_{0}^{1/M} \{(\log a)^{1/2} + (\log M)^{1/2}\}d(1/M)  \\
&\leq& C_6 \frac{1}{M}\{(\log a)^{1/2}+ (\log M)^{1/2}\}
\end{eqnarray*}
for some constant $C_6 > 0$.
Plugging in the value of $N(1/M, T, \norm{\cdot})$ and the bound for $\nu(M)$  in \ref{eq:comp1}, we get the required bound for
$P(\norm{W^J}_{\infty} > M)$.
\end{proof}

\subsection{Proof of Lemma \ref{lem:maxent}}
\begin{proof}
From \cite{mead1984maximum}, for any $k \geq 1$, we can get an infinitely smooth density $f_{m_{\sigma}}$ supported on $[-a_{\sigma}, a_{\sigma}]$ such that
\begin{eqnarray}\label{eq:app1}
\int_{-a_{\sigma}}^{a_{\sigma}}x^jf_{m_{\sigma}}(x)dx = \int_{-a_{\sigma}}^{a_{\sigma}}x^jf(x)dx, \, j=0, \ldots, k.
\end{eqnarray}
One possible choice of $f_{m_{\sigma}}$ in \ref{eq:app1} has the form $f_{m_{\sigma}}(x) = \exp(-\sum_{l=1}^k b_l x^l)$ which corresponds to the maximum entropy moment matching (MAXENT) density. We shall choose $k$ sufficiently large depending on $\sigma$ so that one has the desired approximation result.

Consider an interval $I_{\sigma} = [-(a_{\sigma} + t_{\sigma}), (a_{\sigma} + t_{\sigma})]$ containing the interval
$[-a_{\sigma}, a_{\sigma}]$ for some $t_{\sigma}>0$ to be chosen later depending on $\sigma$.
Observe that
\begin{eqnarray}\label{eq:mead1}
\sup_{x \in I_{\sigma}^c} \abs{\phi_{\sigma}*f(x) - \phi_{\sigma}*f_{m_{\sigma}}(x)} \leq
\sup_{x \in I_{\sigma}^c} \int_{-a_{\sigma}}^{a_{\sigma}} \phi_{\sigma}(x-y)\abs{f(y) - f_{m_{\sigma}}(y)}dy \leq 2\phi_{\sigma}(t_{\sigma}).
\end{eqnarray}
Next, along the lines of \citet{ghosal2007posterior}
\small
\begin{eqnarray}
\sup_{x \in I_{\sigma}} \abs{\int_{-a_{\sigma}}^{a_{\sigma}} \phi_{\sigma}(x-y)f(y) - f_{m_{\sigma}}(y)dy} &\leq&
2\sup_{x \in I_{\sigma}, \abs{y} \leq a_{\sigma}}\abs{\phi_{\sigma}(x-y) - \sum_{j=0}^{k}\frac{1}{\sqrt{2\pi}}\frac{(-1)^j\sigma^{-(2j+1)}(x-y)^{2j}}{j!}} \nonumber\\
&\leq& \frac{2C_1}{\sigma}\sup_{x \in I_{\sigma}, \abs{y} \leq a_{\sigma}}\bigg(\frac{e|x-y|^2}{2k\sigma^2}\bigg)^k \nonumber\\
&\leq& C_2 \sigma^{-(2k+1)} \bigg(\frac{e}{2}\bigg)^k \frac{(2a_{\sigma}+t_{\sigma})^{2k}}{k^k} \label{eq:mead2},
\end{eqnarray}
\normalsize
for some global constants $C_1, C_2 > 0$.

Now choose $t_{\sigma} = Aa_{\sigma}$ for some constant $A > 0$. Then, from \ref{eq:mead1} and \ref{eq:mead2}, we obtain,
\begin{eqnarray*}
\norm{\phi_{\sigma}*f - \phi_{\sigma}*f_{m_{\sigma}}}_{\infty} \leq \max\bigg[2\phi_{\sigma}(t_{\sigma}),
\frac{C_2}{\sigma} \exp\bigg\{ 2k\log \frac{(2+A)a_0\abs{\log \sigma}^{1/\tau_2}}{\sigma} - k\log \frac{k}{B} \bigg\} \bigg],
\end{eqnarray*}
where $B= e/2$.
Choosing $k= B\sigma^{-\alpha}\abs{\log \sigma}^{\alpha/\tau_2}$,
\begin{eqnarray*}
k\log \frac{k}{B} =B\sigma^{-\alpha}\abs{\log \sigma}^{\alpha/\tau_2} \{ \alpha \log (1/\sigma) + (\alpha / \tau_2)\log (\abs{\log \sigma})\},
\end{eqnarray*}
and
\begin{eqnarray*}
2k\log \frac{(2+A)a_0\abs{\log \sigma}^{1/\tau_2}}{\sigma} =2B\sigma^{-\alpha}\abs{\log \sigma}^{\alpha/\tau_2} \{ \log\{(2+A)a_0\} + (1 / \tau_2)\log (\abs{\log \sigma})\}.
\end{eqnarray*}
Clearly, if $\alpha \geq 2$ and $\sigma$ is sufficiently small, $2k\log \frac{(2+A)a_0\abs{\log \sigma}^{1/\tau_2}}{\sigma} < k\log \frac{k}{B}$.  Then by choosing $\alpha > 2$, we can make $2\phi_{\sigma}(t_{\sigma}) > \frac{C_2}{\sigma} \exp\bigg\{ 2k\log \frac{(2+A)a_0\abs{\log \sigma}^{1/\tau_2}}{\sigma} - k\log \frac{k}{B}\bigg\}$ and hence
\begin{eqnarray}
\norm{\phi_{\sigma}*f - \phi_{\sigma}*f_{m_{\sigma}}}_{\infty} \leq 2\phi_{\sigma}(t_{\sigma}) = \frac{C_3}{\sigma} \exp\{-(a_0A)^2/2\abs{\log \sigma}^{2/\tau_2}\}.
\end{eqnarray}
Since $A$ is arbitrary, the conclusion of the theorem follows.

\end{proof}

\subsection{Proof of Lemma \ref{lem:priorconc}}
\begin{proof}
 $m_{\sigma}^{J}$ is contained in the RKHS of $W^J$,
\begin{align*}
& \inf\{\norm{w}^2_{\mathbb{H}}: \norm{w - m_{\sigma}^J}_{\infty} < \epsilon\} = \norm{m_{\sigma}^J}_{\mathbb{H}}^2 = \sum_{j=0}^{J}\frac{m_{\sigma,j}^2}{\lambda_j^2} \\ & = \sum_{j=0}^{J}j^{q/2}m_{\sigma,j}^2 \leq \sum_{j=0}^{J}j^qm_{\sigma,j}^2  \precsim  \frac{a^q}{b^2}\norm{m_{\sigma}^{J}}^{2}_{2,q}.
\end{align*}

Next we calculate $\mbox{P}(\norm{W^J}_{\infty} \leq \epsilon)$ using a technique similar to the proof on Theorem 4.5 in \cite{van2008rates}.
For any numbers $\alpha_j \geq 0$ with $\sum_{j=0}^{J}\alpha_j \leq 1$,we have
\begin{eqnarray*}
\mbox{P}(\norm{W^J}_{\infty} \leq \epsilon) &\geq& \mbox{P}(\sum_{j=0}^{J}\abs{Z_j\lambda_jb} < \epsilon) \\
&\geq& \prod_{j=0}^{J} \mbox{P}( \abs{Z_j\lambda_j} < \alpha_j\epsilon/b).
\end{eqnarray*}
Now, define a function $f : [0, \infty) \to \mathbb{R}$ given by $f(y) = - \log \mbox{P}(\abs{Z} < y) = - \log \{2 \Phi(y) - 1 \}$, where $Z \sim \mbox{N}(0,1)$. $f$ is a decreasing function and following \cite{van2008rates}, $f$ is bounded above by a multiple of $1 + \abs{\log y}$
for $y \in [0, c]$ and bounded above by a multiple of $e^{-y^2/2}$ for $y \geq c$ for some $c > 0$.
Thus, with $\alpha_j = (K + j^2)^{-1}$ for a large constant $K > 0$,
\begin{eqnarray*}
- \log P(\norm{W^J}_{\infty} \leq \epsilon) &\leq& \sum_{j=1}^{J} f(\alpha_j\epsilon j^{q/4}/b) + f(\epsilon/(Kb)) \\
&\leq & \int_{1}^{J} f\bigg(\frac{\epsilon x^{q/4}}{b(K + x^2)} \bigg) dx + f(\epsilon/b(K+1))   + f(\epsilon/(Kb)) ,
\end{eqnarray*}
where the last inequality in the above display follows from the fact that $f$ is decreasing and the map
$x \mapsto x^{q/4}/(K+x^2)$ is non-decreasing on $[1, \infty)$ for any $K > 0$ as long as $q > 4$. For $\epsilon$ small enough so that $\epsilon/(Kb) < c$, $f(\epsilon/(Kb)) < 1 + \log(Kb/\epsilon)$.

Now consider two cases to bound the integral in the last display. If $\epsilon J^{q/4} \leq (K + J^2)$,  $\epsilon x^{q/4}/(K + x^2) \leq 1$ for $x \in [1, J]$. Hence in that case,
\begin{align}
&\int_{1}^{J} f\bigg(\frac{\epsilon x^{q/4}}{b(K + x^2)} \bigg) dx \leq \int_{1}^{J} \bigg( 1 + \abs{\log \bigg(\frac{\epsilon x^{q/4}}{b(K + x^2)}\bigg)}\bigg) dx \\
& \leq \int_{1}^{J}\bigg(1 + \log\frac{b(K+1)}{\epsilon} \bigg)dx
\leq J \bigg(1 + \log\frac{b(K+1)}{\epsilon}\bigg).
\end{align}
On the other hand, if $\epsilon J^{q/4}/b > (K + J^2)$,
\begin{eqnarray}
\int_{1}^{J} f\bigg(\frac{\epsilon x^{q/4}}{K + x^2} \bigg) dx = \bigg(\frac{b}{\epsilon }\bigg)^{4/q} \times (4/q) \times \int_{\epsilon/b}^{\epsilon J^{q/4}/b} f\bigg(\frac{ y}{K + (by/\epsilon)^{8/q}}\bigg)y^{4/q -1 } dy.
\end{eqnarray}
The integral above is bounded by
\begin{align}
&\int_{0}^{b/\epsilon} f\bigg(\frac{y}{K + (b/\epsilon)^{16/q}} \bigg) y^{4/q -1 }dy + \int_{b/\epsilon}^{\infty}f\bigg(\frac{y}{K + y^{16/q}} \bigg) y^{4/q -1 }dy \\
&\leq z^{4/q}\int_{0}^{b/(z\epsilon)}f(x) x^{4/q -1} dx  + \int_{0}^{\infty}f\bigg(\frac{y}{K + y^{16/q}} \bigg) y^{4/q -1 }dy
\end{align}
for $z = ( K + (b/\epsilon)^{16/q})$.
The first integral is bounded as $\epsilon \downarrow 0$ and the second integral is finite for $q > 16$.
Hence for $q > 16$,
\begin{eqnarray*}
\phi_{m_{\sigma}^{J}}(\epsilon) \precsim  \frac{1}{a^qb^2}\norm{m_{\sigma}^{J}}_{2,q}^2 +
\begin{cases}
J \big(1 + \log \frac{b}{\epsilon}\big), \, \epsilon J^{q/4} \precsim bJ^2 \\
\big(\frac{b}{\epsilon}\big)^{20/q},  \,  \epsilon J^{q/4} \geq bJ^2
\end{cases}
\end{eqnarray*}
\end{proof}

\bibliographystyle{imsart-nameyear}
\bibliography{Xbib19}

\end{document}